\documentclass{amsart}

\usepackage[all]{xy}
\usepackage{amssymb,latexsym}
\usepackage{mltex}
\usepackage{amsmath}
\usepackage[colorlinks=true,linkcolor=blue,citecolor=blue]{hyperref}
 \usepackage{lscape}
\usepackage{enumerate}
\usepackage{amsthm}
\usepackage{hyperref}
\usepackage{multicol}

\newtheorem{thm}{Theorem}
\newtheorem{cor}{Corollary}
\newtheorem{lem}{Lemma}[section]
\newtheorem{prop}[lem]{Proposition}
\newtheorem{rem}{Remark}

\newcommand{\C}{\mathbb{C}}

\newcommand{\HH}{\mathbb{H}}

\newcommand{\R}{\mathbb{R}}

\newcommand{\beqt}{\begin{equation}}  \newcommand{\eeqt}{\end{equation}}
\newcommand{\bal}{\begin{align}}      \newcommand{\eal}{\end{align}}
\newcommand{\ba}{\begin{array}}      \newcommand{\ea}{\end{array}}
\newcommand{\bc}{\begin{center}}     \newcommand{\ec}{\end{center}}
\newcommand{\be}{\begin{enumerate}}  \newcommand{\ee}{\end{enumerate}}
\newcommand{\beq}{\begin{eqnarray}}  \newcommand{\eeq}{\end{eqnarray}}
\newcommand{\beQ}{\begin{eqnarray*}} \newcommand{\eeQ}{\end{eqnarray*}}
\newcommand{\bi}{\begin{itemize}}    \newcommand{\ei}{\end{itemize}}
\newcommand{\bt}{\begin{tabular}}    \newcommand{\et}{\end{tabular}}

\begin{document}
\title{Spinorial representation of submanifolds in a product of space forms}
\author{Alicia Basilio, Pierre Bayard,  Marie-Am\'elie Lawn, Julien Roth }
\address{Alicia Basilio. Facultad de Ciencias, UNAM, M\'exico.}
\email{aliciabasilo.ab@gmail.com}
\address{Pierre Bayard. Facultad de Ciencias, UNAM, M\'exico.}
\email{bayard@ciencias.unam.mx}
\address{Marie-Am\'elie Lawn. Department of Mathematics, Imperial College, United Kingdom.} 
\email{marieamelie.lawn@gmail.com}
\address{Julien Roth. Lab. d'Analyse et de Math\'ematiques Appliqu\'ees, U. Eiffel, France.} 
\email{julien.roth@univ-eiffel.fr}

\maketitle


\begin{abstract}
We present a method giving a spinorial characterization of an immersion in a product of spaces of constant curvature. As a first application we obtain a proof using spinors of the fundamental theorem of immersion theory in that spaces. We also study special cases: we recover previously known results concerning immersions in $\mathbb{S}^2\times\R$ and we obtain new spinorial characterizations of immersions in $\mathbb{S}^2\times\R^2$ and in $\mathbb{H}^2\times\mathbb{R}.$ We then study the theory of $H=1/2$ surfaces in $\mathbb{H}^2\times\R$ using this spinorial approach, obtain new proofs of some of its fundamental results and give a direct relation with the theory of $H=1/2$ surfaces in $\R^{1,2}.$
\end{abstract}

\date{}
\maketitle\pagenumbering{arabic}
\section*{Introduction}
Characterizations of immersions in space forms using spinors have been widely studied, as for instance in \cite{Fr,Ko,KS,Mo,Ta,Tr,Va,Vo} and more recently in \cite{BLR2} (see also the references in these papers). It appears that spin geometry furnishes an elegant formalism for the description of the immersion theory in space forms, especially in low dimension and in relation with the Weierstrass representation formulas. See also the Weierstrass representation obtained in \cite{AMM} for CMC hypersurfaces in some four-dimensional Einstein manifolds. We are interested here in spinorial characterizations of immersions in a product of space forms. Some special cases have been studied before, as immersions in the products $\mathbb{S}^2\times\R,$ $\mathbb{H}^2\times\R,$ $\mathbb{S}^2\times\R^2$ and $\mathbb{S}^3\times\R$  \cite{LR,Ro1,Ro2}. We propose here a method allowing the treatment of an immersion in an arbitrary product of space forms. The dimension and the co-dimension of the immersion are moreover arbitrary. Let us note that the spinor bundle that we use in the paper is not the usual spinor bundle: in general it is a real bundle, and of larger rank. We used this idea in \cite{BLR2}. Let us also mention that even in low dimensions we obtain new results: the theory permits to recover in a unified way the previously known results and to complete them; in particular, we show how to recover the spinorial characterization of an immersion in $\mathbb{S}^2\times\R$ and we obtain new spinorial characterizations of immersions in $\mathbb{S}^2\times\R^2$ and in $\mathbb{H}^2\times\mathbb{R}.$ 

A first application of the general theory is a proof using spinors of the fundamental theorem of immersion theory in a product of space forms.

A second application concerns the theory of CMC surfaces with $H=1/2$ in $\mathbb{H}^2\times\R$: we show that a component of the spinor field representing the immersion of such a surface is an horizontal lift of the hyperbolic Gauss map, for a connection which depends on the Weierstrass data of the immersion, and we deduce that there exists a two-parameter family of $H=1/2$ surfaces in $\HH^2\times\R$ with given hyperbolic Gauss map and Weierstrass data, a result obtained in \cite{FM1} using different methods. We finally study the spinorial representation of $H=1/2$ surfaces in $\R^{1,2}$ and obtain a direct relation between the two theories.

In order to simplify the exposition we first consider immersions in a product of spheres $\mathbb{S}_1^m\times \mathbb{S}_2^n$ and in a product $\mathbb{S}_1^m\times \mathbb{R}^n,$ and we then state without proof the analogous results for immersions in a product of hyperbolic spaces $\mathbb{H}_1^m\times \mathbb{H}_2^n$ and in $\mathbb{H}_1^m\times \mathbb{R}^n.$ Using the same ideas it is possible to state analogous results for an arbitrary quantity of factors involving $\mathbb{S}_1^m,$ $\mathbb{H}_2^n$ and $\mathbb{R}^p$, or for space forms with pseudo-riemannian metrics, but these general statements are not included in the paper. 

The outline of the paper is as follows. We first study the immersions in a product of spheres $\mathbb{S}_1^m\times\mathbb{S}_2^n$ in Section \ref{section main result SmSn} and the immersions in $\mathbb{S}_1^m\times\R^n$ in Section \ref{section SmRn}. We then state the analogous results for a product of hyperbolic spaces $\mathbb{H}_1^m\times \mathbb{H}_2^n$ and for $\mathbb{H}_1^m\times \mathbb{R}^n$ in Section \ref{section HmHn HmRn}. Finally the theory of $H=1/2$ surfaces in $\mathbb{H}^2\times\R$ is studied in Section \ref{section H2R}. Some useful auxiliary results are gathered in an appendice at the end of the paper.

\section{Isometric immersions in $\mathbb{S}_1^m\times \mathbb{S}_2^n$}\label{section main result SmSn}
We are interested here in immersions in a product $\mathbb{S}_1^m\times \mathbb{S}_2^n$ of two spheres, of constant curvature $c_1,c_2>0$. We construct the suitable spinor bundle in Section \ref{section spinor bundle}, we consider the case of a manifold which is already immersed in $\mathbb{S}_1^m\times\mathbb{S}_2^n$ in Section \ref{section spin SmSn}, we state and prove the main theorem in Section \ref{section main thm SmSn} and the fundamental theorem in Section \ref{section fal thm SmSn}. 
\subsection{The suitable spinor bundle}\label{section spinor bundle}
Let $M$ be a $p$-dimensional riemannian manifold and $E\rightarrow M$ a vector bundle of rank $q,$ with $p+q=m+n,$ with a bundle metric and a connection compatible with the metric. Let $\mathcal{E}_2=M\times\R^2\rightarrow M$ be the trivial bundle, equipped with its natural metric and its trivial connection. Let us construct a spinor bundle on $M$ equipped with a Clifford action of $TM\oplus E\oplus \mathcal{E}_2.$ We suppose that $M$ and $E$ are spin, with spin structures $\widetilde{Q}_M\rightarrow Q_M,$ $\widetilde{Q}_E\rightarrow Q_E$ and set $\widetilde{Q}:=\widetilde{Q}_M\times_M\widetilde{Q}_E.$ Let us denote by $Spin(N)$ and $Cl(N)$ the spin group and the Clifford algebra of $\R^N.$ Associated to the splitting $\R^{m+n+2}=\R^p\oplus\R^q\oplus\R^2$ we consider
$$Spin(p)\cdot Spin(q)\subset Spin(p+q)\subset Spin(m+n+2)$$
and define the representation
\begin{eqnarray*}
\rho:\hspace{.5cm}Spin(p)\times Spin(q)&\rightarrow& GL(Cl(m+n+2))\\
a:=(a_p,a_q)&\mapsto&\rho(a):\ (\xi\mapsto a_p\cdot a_q\cdot\xi)
\end{eqnarray*}
with the bundles
$$\Sigma:=\widetilde{Q}\times_{\rho} Cl(m+n+2),\hspace{1cm}U\Sigma:=\widetilde{Q}\times_{\rho}Spin(m+n+2)\subset \Sigma.$$
Since the bundle of Clifford algebras constructed on the fibers of $TM\oplus E\oplus\mathcal{E}_2$ is
$$Cl(TM\oplus E\oplus\mathcal{E}_2)= \widetilde{Q}\times_{Ad} Cl(m+n+2)$$ 
with
\begin{eqnarray*}
Ad:\hspace{.5cm}Spin(p)\times Spin(q)&\rightarrow& GL(Cl(m+n+2))\\
a&\mapsto&Ad(a):\ (\xi\mapsto a\cdot\xi\cdot a^{-1}),
\end{eqnarray*}
there is a Clifford action 
\begin{eqnarray*}
Cl(TM\oplus E\oplus\mathcal{E}_2)\ \oplus\ \Sigma&\rightarrow&\Sigma\\
(Z,\varphi)&\mapsto&Z\cdot\varphi
\end{eqnarray*} 
similar to the usual Clifford action in spin geometry. Let us note that $\Sigma$ is not the usual spinor bundle, since it is a real vector bundle, associated to a representation which is not irreducible: it is rather a (maybe large) sum of real spinor bundles. We nevertheless interpret the bundle $\Sigma$ as \emph{the bundle of spinors}, $U\Sigma$ as \emph{the bundle of unit spinors} and $Cl(TM\oplus E\oplus\mathcal{E}_2)$ as \emph{the Clifford bundle} acting on the bundle of spinors. There is a natural map
\begin{eqnarray*}
\langle\langle.,.\rangle\rangle:\hspace{.5cm}\Sigma\times\Sigma&\rightarrow&Cl(m+n+2)\\
(\varphi,\varphi')&\mapsto&\langle\langle\varphi,\varphi'\rangle\rangle:=\tau[\varphi']\ [\varphi]
\end{eqnarray*}
where $\varphi=[\tilde{s},[\varphi]]$ and $\varphi'=[\tilde{s},[\varphi']]$ in $\Sigma=\widetilde{Q}\times_{\rho} Cl(m+n+2)$ and $\tau$ is the involution of $Cl(m+n+2)$ reversing the order of a product of vectors. Here and in all the paper we use the brackets $[.]$ to denote the component in $Cl(m+n+2)$ of an element of the spinor or the Clifford bundle in a given spinorial frame $\widetilde{s}.$ This map is such that, for all $\varphi,\varphi'\in\Sigma$ and $Z\in TM\oplus E\oplus\mathcal{E}_2,$ 
\begin{equation}\label{prop prod sigma 1}
\langle\langle\varphi,\varphi'\rangle\rangle=\tau\langle\langle\varphi',\varphi\rangle\rangle
\end{equation}
and
\begin{equation}\label{prop prod sigma 2}
\langle\langle Z\cdot\varphi,\varphi'\rangle\rangle=\langle\langle\varphi,Z\cdot \varphi'\rangle\rangle.
\end{equation}
Moreover, it is compatible with the connection $\nabla$ induced on $\Sigma$ by the Levi-Civita connection on $M$ and the given connection on $E:$
\begin{lem}
For all $X\in TM$ and $\varphi,\varphi'\in\Gamma(\Sigma),$
\begin{equation}\label{prop prod sigma 3}
\partial_X \langle\langle\varphi,\varphi'\rangle\rangle=\langle\langle\nabla_X\varphi,\varphi'\rangle\rangle+ \langle\langle\varphi,\nabla_X\varphi'\rangle\rangle
\end{equation}
where on the left hand side $\partial$ stands for the usual derivative.
\end{lem}
A similar result is proved in \cite[Lemma 2.2]{BLR2}.
\subsection{Spin geometry of a submanifold in $\mathbb{S}_1^m\times \mathbb{S}_2^n$}\label{section spin SmSn}
We assume in that section that $M$ is a $p$-dimensional submanifold of $\mathbb{S}_1^m\times\mathbb{S}_2^n,$ with normal bundle $E$ of rank $q$ and second fundamental form $B:TM\times TM\rightarrow E,$ denote by $\nu_1:M\rightarrow\mathbb{R}^{m+1}$ and $\nu_2:M\rightarrow\mathbb{R}^{n+1}$ the vector fields such that $\frac{1}{\sqrt{c_1}}\nu_1$ and $\frac{1}{\sqrt{c_2}}\nu_2$ are the two components of the immersion $M\rightarrow \mathbb{S}_1^m\times\mathbb{S}_2^n$ and consider the trivial bundle $\mathcal{E}_2=\R\nu_1\oplus\R\nu_2\rightarrow M.$ We consider spin structures on $TM$ and $E,$ and the bundles $\Sigma,$ $U\Sigma$ and $Cl(TM\oplus E\oplus\mathcal{E}_2)$ constructed in the previous section. For a convenient choice of the spin structures on $TM$ and $E,$ the bundle $\Sigma$ identifies canonically with the trivial bundle $M\times Cl(m+n+2)$, and two connections are defined on $\Sigma,$ the connection $\nabla$ introduced above and the trivial connection $\partial.$ Since the second fundamental form of $M$ in $\R^{m+n+2}$ is
$$(X,Y)\mapsto -\sqrt{c_1}\ \langle X_1,Y_1\rangle \nu_1 - \sqrt{c_2}\ \langle X_2,Y_2\rangle \nu_2+B(X,Y)$$
where $X=X_1+X_2$ and $Y=Y_1+Y_2$ in the decomposition $TM\subset T\mathbb{S}_1^m\oplus T\mathbb{S}_2^n$ and setting 
\begin{equation}\label{def B Clifford}
\frac{1}{2}B(X):=\frac{1}{2}\sum_{j=1}^pe_j\cdot B(X,e_j)\ \in Cl(TM\oplus E\oplus \mathcal{E}_2)
\end{equation}
where $e_1,\ldots, e_p$ is an orthonormal basis of $TM$, they satisfy the following Gauss formula:
\begin{equation}\label{gauss equation}
\partial_X\varphi=\nabla_X \varphi-\frac{1}{2} (\sqrt{c_1}\ X_1 \cdot \nu_1  +  \sqrt{c_2}\ X_2 \cdot \nu_2)\cdot\varphi  +  \frac{1}{2}B(X)\cdot \varphi
\end{equation}
for all $\varphi\in\Gamma(\Sigma)$ and all $X\in TM$; see \cite{Bar} for the proof of a spinorial Gauss formula in a slightly different context. By formula (\ref{gauss equation}), the constant spinor field $\varphi={1_{Cl(m+n+2)}}_{|M}$ satisfies, for all $X\in TM,$
\begin{equation}\label{spin gauss equation SmSn}
\nabla_X \varphi = \frac{1}{2} (\sqrt{c_1}\ X_1 \cdot \nu_1  +  \sqrt{c_2}\ X_2 \cdot \nu_2)\cdot\varphi  -  \frac{1}{2}B(X)\cdot \varphi.
\end{equation}
\subsection{The main theorem}\label{section main thm SmSn}
We assume that $M$ and $E\rightarrow M$ are abstract objects as in Section \ref{section spinor bundle} (i.e. $M$ is not a priori immersed in $\mathbb{S}_1^m\times \mathbb{S}_2^n$), and suppose moreover that there is a product structure on $TM\oplus E,$ i.e. a bundle map $\mathcal{P}:TM\oplus E\rightarrow TM\oplus E$ such that $\mathcal{P}^2=id,$ $\mathcal{P}\neq id.$ Setting $\mathcal{P}_1:=\mbox{Ker} (\mathcal{P}-id)$ and $\mathcal{P}_2:=\mbox{Ker} (\mathcal{P}+id)$ we have $TM\oplus E=\mathcal{P}_1\oplus \mathcal{P}_2:$ the product structure $\mathcal{P}$ is equivalent to a splitting of $TM\oplus E$ into two subbundles $\mathcal{P}_1$ and $\mathcal{P}_2,$ that we assume to be respectively of rank $m$ and $n.$

\subsubsection{Statement of the theorem}\label{subsection main SmSn}
Let $B: TM \times TM \to E$ be a symmetric tensor. Let us fix two unit orthogonal and parallel sections $\nu_1,\nu_2$ of the trivial bundle $\mathcal{E}_2=M\times\R^2\rightarrow M.$
\begin{thm}\label{th main result SmSn}
The following statements are equivalent:
\begin{enumerate}
\item[(i)] There exist an isometric immersion $F:M\rightarrow \mathbb{S}_1^m\times \mathbb{S}_2^n$ and a bundle map $\Phi: TM \oplus E \to T(\mathbb{S}_1^m \times \mathbb{S}_2^n)$ above $F$ such that $\Phi(X,0)= dF(X)$ for all $X\in TM,$ which preserves the bundle metrics, maps the connection on $E$ and the tensor $B$ to the normal connection and the second fundamental form of $F$, and  is compatible with the product structures.
\\
\item[(ii)] \label{thm main (2)} There exists a section $\varphi \in \Gamma(U \Sigma )$ solution of
\begin{equation}\label{killing equation SmSn}
\nabla_X \varphi = \frac{1}{2} (\sqrt{c_1}\ X_1 \cdot \nu_1  +  \sqrt{c_2}\ X_2 \cdot \nu_2)\cdot\varphi  -  \frac{1}{2}B(X)\cdot \varphi
\end{equation}
for all $X\in TM,$ where $X= X_1 + X_2$ is the decomposition in the product structure $\mathcal{P}$ of $TM\oplus E$, such that the map
$$Z\in TM\oplus E\ \mapsto\ \langle\langle Z \cdot \varphi , \varphi \rangle\rangle\in\R^{m+1}\times\R^{n+1}$$
commutes with the product structure $\mathcal{P}$ and the natural product structure on $\R^{m+1}\times\R^{n+1}.$
\end{enumerate}
\vspace{.5cm}

Moreover, the bundle map $\Phi$ and the immersion $F$ are explicitly given in terms of the spinor field $\varphi$ by the formulas
$$\Phi:\hspace{.3cm}TM \oplus E \to T(\mathbb{S}_1^m \times \mathbb{S}_2^n),\hspace{.3cm}Z \mapsto \langle\langle Z \cdot \varphi , \varphi \rangle\rangle$$
and
\begin{equation}\label{explicit f SmSn}
F=(\frac{1}{\sqrt{c_1}} \langle\langle \nu_1 \cdot \varphi ,  \varphi \rangle\rangle , \frac{1}{\sqrt{c_2}}\langle\langle\nu_2 \cdot \varphi , \varphi   \rangle\rangle )\hspace{.3cm} \in\ \mathbb{S}_1^m\times\mathbb{S}_2^n.
\end{equation}
\end{thm}
\begin{rem}
Formulas (\ref{killing equation SmSn}) and (\ref{explicit f SmSn}) can be regarded as a generalized Weierstrass representation formula.
\end{rem}
\subsubsection{Proof of Theorem \ref{th main result SmSn}} The proof of $"(i)\Rightarrow(ii)"$ was obtained in Section \ref{section spin SmSn}: if $M$ is immersed in $\mathbb{S}_1^m\times\mathbb{S}_2^n,$ the spinor field $\varphi$ is the constant spinor field $1_{Cl(m+n+2)}$ restricted to $M.$ We prove $"(ii)\Rightarrow(i)".$
We suppose that $\varphi \in \Gamma( U \Sigma )$  is a solution of (\ref{killing equation SmSn}) and obtain $(i)$ as a direct consequence of the following two lemmas:
\begin{lem}\label{lem 1 pf thm SmSn}
The map $F$ defined by (\ref{explicit f SmSn})  satisfies  
\begin{equation}\label{eqn df SmSn}
dF(X) = \langle\langle X \cdot \varphi , \varphi \rangle\rangle,
\end{equation}
for all $X\in TM.$ It preserves the product structures and takes values in $\mathbb{S}_1^m \times \mathbb{S}_2^n.$
\end{lem} 
\begin{proof}
Let us consider for $i=1,2$ the functions $F_i =\frac{1}{\sqrt{c_i}}\langle\langle \nu_i \cdot \varphi , \varphi \rangle\rangle.$ Recalling the properties (\ref{prop prod sigma 1})-(\ref{prop prod sigma 3}) of $\langle\langle .,.\rangle\rangle$ and since $\nu_1,\nu_2$ are parallel sections of $\mathcal{E}_2$ and $\varphi$ satisfies \eqref{killing equation SmSn}, we have, for $i=1,2,$
\begin{eqnarray}
dF_i(X) &=& \frac{1}{\sqrt{c_i}}( \langle\langle  \nu_i \cdot \nabla_X   \varphi , \varphi  \rangle\rangle + \langle\langle \nu_i \cdot \varphi , \nabla_X\varphi \rangle\rangle) \nonumber \\
&=& \frac{1}{2\sqrt{c_i}} (\tau + id) \langle\langle \nu_i \cdot (-B(X) + \sqrt{c_1}\ X_1 \cdot  \nu_1 + \sqrt{c_2}\ X_2 \cdot \nu_2 ) \cdot \varphi , \varphi \rangle\rangle \nonumber \\
&=& \langle\langle X_i \cdot \varphi , \varphi \rangle\rangle  \label{eqn dfiXi}
\end{eqnarray}
where we use in the last equality that, for all $i,j\in\{1,2\},$
$$\tau(\nu_i\cdot B(X))=-B(X)\cdot\nu_i=-\nu_i \cdot B(X)$$
and
\begin{eqnarray*}
\tau\left(\nu_i\cdot X_j\cdot\nu_j\right)=\nu_j\cdot X_j\cdot\nu_i&=& - \nu_i \cdot X_j\cdot\nu_j \hspace{.3cm}\mbox{if }i\neq j \\
&=& X_i\hspace{.3cm}\mbox{if } i=j
\end{eqnarray*} 
(since $\nu_i,$ $X_j,$ $\nu_j$ are three orthogonal vectors if $i\neq j$). Since $F=F_1+F_2$ and $X=X_1+X_2,$ (\ref{eqn df SmSn}) follows from (\ref{eqn dfiXi}). Let us see now that $F$ takes values in $\mathbb{S}_1^m \times \mathbb{S}_2^n \subset \R^{m+n+2}$. We assume that $\nu_1,\nu_2$ are given in an arbitrary frame $\widetilde{s}\in\widetilde{Q}$ by $\nu_1=[\widetilde{s},e_{1}^o]$ and $\nu_2=[\widetilde{s},e_{2}^o]$ where $e_1^o$ and $e_2^o$ are the last two vectors of the canonical basis of $\R^{m+n+2}=\R^{p+q}\oplus\R^2.$
Since $[\varphi]$ belongs to $Spin(p+q+2)$, we have 
$$\langle\langle \nu_i \cdot \varphi , \varphi \rangle\rangle = Ad_{[\varphi]}(e_{i}^o)\ \mbox{ for}\ i = 1,2$$
which implies that $F_1$ and $F_2$ take values in the spheres of $\R^{m+n+2}$ of radius $1/\sqrt{c_1}$ and $1/\sqrt{c_2}$ respectively. We then have to check that $F_1$ and $F_2$ take respectively values in $\R^{m+1}$ and $\R^{n+1}:$ since $dF$ preserves the product structures, we have for $X=X_1+X_2 \in TM$ that
$$dF(X)  = dF(X_1)+dF(X_2)\hspace{.5cm} \in\ \mathbb{R}^{m+1}\oplus \mathbb{R}^{n+1};$$ 
since by (\ref{eqn dfiXi}) $dF_1(X)=dF(X_1)$ and $dF_2(X)=dF(X_2)$ we conclude that $dF_1$ and $dF_2$ take values in $\R^{m+1}$ and $\R^{n+1}$ respectively, and so do $F_1$ and $F_2.$
\end{proof}
\begin{lem} The map
$$\widetilde{\Phi}:\hspace{.3cm} TM \oplus E\oplus\mathcal{E}_2 \to \R^{m+1} \times\R^{n+1},\hspace{.3cm}Z \mapsto \langle\langle Z \cdot \varphi , \varphi \rangle\rangle$$
is a bundle map which preserves the metrics, identifies $E$ with the normal bundle of the immersion $F$ in $\mathbb{S}_1^m \times \mathbb{S}_2^n$, and sends the connection on $E$ and the tensor $B$ to the normal connection and the second fundamental form of the immersion $F$.
\end{lem}
\begin{proof}
Let us first see that $\Phi$ takes values in $T(\mathbb{S}_1^m \times \mathbb{S}_2^n).$ By Lemma \ref{lem 1 pf thm SmSn}, if $X\in TM,$ $\Phi(X)$ belongs to $T(\mathbb{S}_1^m \times \mathbb{S}_2^n).$ For  $Z_1 , Z_2 \in TM \oplus E \oplus \mathcal{E}_2,$ we have in $Cl(m+n+2)$
\begin{eqnarray*}
\langle  \widetilde{\Phi}(Z_1) , \widetilde{\Phi} (Z_2) \rangle &=& - \frac{1}{2} ( \widetilde{\Phi}(Z_1) \widetilde{\Phi}(Z_2) + \widetilde{\Phi}(Z_2) \widetilde{\Phi}(Z_1)) \\
&=& - \frac{1}{2} ( \tau [\varphi] [Z_1] [\varphi] \tau [\varphi][Z_2][\varphi] +  \tau [\varphi] [Z_2] [\varphi] \tau [\varphi][Z_1][\varphi] ) \\
&=& \tau[\varphi] \langle [Z_1] , [Z_2] \rangle [\varphi] \\
&=& \langle Z_1 , Z_2 \rangle
\end{eqnarray*}  
since $\tau[\varphi][\varphi]=1$ and $\langle Z_1 , Z_2 \rangle = \langle [Z_1] , [Z_2] \rangle$ belongs to $\R.$ Let us note that $\widetilde{\Phi}(\nu_1)$ and $\widetilde{\Phi}(\nu_2)$ are normal to $T(\mathbb{S}_1^m \times \mathbb{S}_2^n)$: for $i=1,2,$ we have at $p\in M$
$$\widetilde{\Phi}({\nu_i}_p) = \langle\langle \nu_i \cdot \varphi , \varphi \rangle\rangle_{F(p)} =\sqrt{c_i}\ F_i(p),$$
which is the unit normal to $\mathbb{S}_i$ at $p.$ Thus, since $\widetilde{\Phi}$ preserves the metric, for $Y\in E,$ $\Phi(Y)$ belongs to $T(\mathbb{S}_1^m \times \mathbb{S}_2^n).$ Let us now compute, for $Z\in\Gamma(TM\oplus E),$
$$\nabla^{\mathbb{S}_1^m \times \mathbb{S}_2^n}_X \Phi(Z) =p_{\mathbb{S}_1^m \times \mathbb{S}_2^n} ( \partial_X \langle\langle Z \cdot \varphi , \varphi \rangle\rangle )$$
where $p_{\mathbb{S}_1^m \times \mathbb{S}_2^n}$ is the projection $\R^{m+1}\times\R^{n+1}\rightarrow T(\mathbb{S}_1^m \times \mathbb{S}_2^n$). We have
$$\partial_X \langle\langle Z \cdot \varphi , \varphi \rangle\rangle=\langle\langle \nabla_XZ \cdot \varphi , \varphi \rangle\rangle+(id+\tau)\langle\langle \varphi , Z \cdot\nabla_X \varphi \rangle\rangle.$$
We focus on the second term. From the Killing type equation (\ref{killing equation SmSn}), we have
$$\langle\langle \varphi , Z \cdot\nabla_X \varphi \rangle\rangle=\frac{1}{2}\sum_{i=1}^2\sqrt{c_i}\ \langle\langle \varphi , Z \cdot X_i\cdot\nu_i\cdot\varphi \rangle\rangle-\frac{1}{2} \langle\langle \varphi , Z \cdot B(X)\cdot\varphi \rangle\rangle.$$
For $i=1,2,$ we have
$$Z\cdot X_i\cdot\nu_i+\nu_i\cdot X_i\cdot Z=\nu_i\cdot\left(Z\cdot X_i+X_i\cdot Z\right)=-2\langle X_i,Z\rangle \nu_i$$
and thus
\begin{eqnarray*}
(id+\tau)\langle\langle \varphi , Z \cdot X_i\cdot\nu_i\cdot\varphi \rangle\rangle&=&\langle\langle \varphi , (Z \cdot X_i\cdot\nu_i+\nu_i\cdot X_i\cdot Z)\cdot\varphi \rangle\rangle\\
&=&-2\langle X_i,Z\rangle \langle\langle \varphi ,\nu_i\cdot\varphi \rangle\rangle.
\end{eqnarray*}
Moreover, we have
$$(id+\tau)\langle\langle \varphi , Z\cdot B(X)\cdot\varphi \rangle\rangle= \langle\langle \varphi , (Z\cdot B(X)-B(X)\cdot Z)\cdot\varphi \rangle\rangle.$$
Since
$$\frac{1}{2}\left(B(X)\cdot Z-Z\cdot B(X)\right)=B(X,Z_T)-B^*(X,Z_N)$$
where $B^*:TM\times E\rightarrow TM$ is so that $\langle B(X,Y),Z\rangle=\langle Y,B^*(X,Z)\rangle$ for all $X,Y\in TM$ and $Z\in E$ (Lemma \ref{lem3 ap1} in Appendix \ref{app clifford}), we deduce that
\begin{eqnarray}
\partial_X \langle\langle Z \cdot \varphi , \varphi \rangle\rangle&=&\langle\langle \nabla_XZ \cdot \varphi , \varphi \rangle\rangle-\sum_{i=1}^2\sqrt{c_i}\ \langle X_i,Z\rangle \langle\langle \varphi ,\nu_i\cdot\varphi \rangle\rangle\label{partial X phi dem thm}\\
&&+ \langle\langle \varphi , (B(X,Z_T)-B^*(X,Z_N))\cdot\varphi \rangle\rangle,\nonumber
\end{eqnarray}
and
$$\nabla_X^{\mathbb{S}_1^m\times\mathbb{S}_2^n}\Phi(Z)=\Phi(\nabla_XZ+B(X,Z_T)-B^*(X,Z_N)).$$
This formula implies the following expressions for the second fundamental form $B^F$ and the normal connection $\nabla'^F$ of the immersion $F:$ if $Z\in\Gamma(TM)$ is such that $\nabla Z=0$ at the point where we do the computations, then $B^F(X,Z)=\Phi(B(X,Z));$ if $Z\in\Gamma(E)$, then $\nabla'^F_X(\Phi(Z))=\Phi(\nabla'_XZ).$ This finishes the proof of the lemma (and of Theorem \ref{th main result SmSn}).
\end{proof}

\subsection{The fundamental theorem in $\mathbb{S}_1^m\times \mathbb{S}_2^n$}\label{section fal thm SmSn}
We give here a proof using spinors of the fundamental theorem of the immersion theory in $\mathbb{S}_1^m\times \mathbb{S}_2^n$. This result has been proved independently by Kowalczyk \cite{K} and Lira-Tojeiro-Vit\'orio \cite{LTV}.
\subsubsection{Statement of the theorem}
Let $\mathcal{P}$ and $\mathcal{P}'$ be the product structures of $TM \oplus E$ and $\R^{m+n+2}=\R^{m+1}\times\R^{n+1}$. We define $f:TM \to TM, h:TM \to E,s:E \to TM$ and $t:E \to E$ such that 
$$\mathcal{P}(X)=\left\{\begin{array}{l} f(X) + h(X)\hspace{.5cm} \mbox{if}\ X \in TM,\\
s(X) + t(X)\hspace{.5cm}\mbox{if}\ X \in E.\end{array}\right.$$
We set, for $U,V,W\in TM\oplus E,$
$$(U\wedge V)W:=\langle U, W\rangle V-\langle V, W\rangle U.$$
We first write the compatibility equations necessary for the existence of a non-trivial spinor field solution of (\ref{killing equation SmSn}):
\begin{prop}\label{prop killing implies fundamental equations}
Let $\varphi \in \Gamma(U\Sigma)$ be a solution of (\ref{killing equation SmSn}) such that 
$$\Phi:\hspace{.3cm} TM\oplus E\to T(\mathbb{S}_1^m\times\mathbb{S}_2^n),\hspace{.3cm}Z\mapsto\langle\langle Z\cdot\varphi,\varphi\rangle\rangle$$
commutes with the product structures, i.e. satisfies
\begin{equation}\label{eqn lem app th fal}
\Phi( \mathcal{P}(Z)) = \mathcal{P}'(\Phi(Z))
\end{equation}
for all $Z \in TM \oplus E.$ If $R^T$ stands for the curvature tensor of the Levi-Civita connection on $M$ and $R^N$ for the curvature tensor of the connection $\nabla'$ on $E,$ the following fundamental equations hold: for all $X,Y,Z\in TM$ and $N\in E,$
\begin{eqnarray}
R^T(X,Y)Z &=& B^*(X,B(Y,Z)) - B^*(Y,B(X,Z))\label{Gauss equation}\\
&& - \frac{1}{4}\left(c_1+c_2\right) \left( (X \wedge Y ) Z + (f(X) \wedge f(Y)) Z \right)\nonumber\\
&& - \frac{1}{4}\left(c_1-c_2\right)  \left( (f(X) \wedge Y )+X\wedge f(Y)\right) Z,\nonumber
\end{eqnarray}
\begin{eqnarray}
R^N(X,Y)N &=& B(X,B^*(Y,N))- B(Y,B^*(X,N))\label{Ricci equation} \\
&&+\frac{1}{4}\left(c_1+c_2\right) \left( h(X) \wedge h(Y) \right) N,\nonumber
\end{eqnarray}
\begin{eqnarray}
(\widetilde{\nabla}_XB)(Y,Z)- (\widetilde{\nabla}_YB)(X,Z) &=& \frac{1}{4}\left(c_1+c_2\right) \left( \langle f(Y), Z \rangle h(X) - \langle f(X) , Z \rangle h(Y) \right)\nonumber\\
&&- \frac{1}{4}\left(c_1-c_2\right)  \left( \langle Y,Z\rangle h(X)-\langle X,Z\rangle h(Y)\right)\label{Codazzi equation}
\end{eqnarray} 
where $\widetilde{\nabla}$ stands for the natural connection on $T^*M\otimes T^*M\otimes E.$ Moreover, if we use the same symbol $\widetilde{\nabla}$ to denote the natural connections on $T^*M\otimes TM,$ $T^*M\otimes E,$ $E^*\otimes TM$ and $E^*\otimes E,$ we have, for $X,Y\in TM$ and $Z\in E,$
\begin{eqnarray}
(\widetilde{\nabla}_Y f)(X)  &=& s (B(Y,X)) + B^*(Y,h(X)), \label{eqn fhst 1}\\
(\widetilde{\nabla}_Y h)(X) &=& t(B(Y,X)) - B(Y,f(X)), \label{eqn fhst 2} \\
(\widetilde{\nabla}_Y s)(Z) &=& -f (B^*(Y,Z)) + B^*(Y,t(Z)), \label{eqn fhst 3} \\
(\widetilde{\nabla}_Y t)(Z) &=& - h(B^*(Y,Z)) - B(Y,s(Z)). \label{eqn fhst 4}
\end{eqnarray} 
\end{prop}
Equations (\ref{Gauss equation}), (\ref{Ricci equation}) and (\ref{Codazzi equation}) are respectively the equations of Gauss, Ricci and Codazzi. Equations (\ref{eqn fhst 1})-(\ref{eqn fhst 4}) are additional equations traducing that $\Phi$ commutes with the product structures $\mathcal{P}$ and $\mathcal{P}',$ with $\mathcal{P}'$ parallel in $\R^{m+n+2}.$ All these equations are necessary for the existence of an immersion $M\rightarrow \mathbb{S}_1^m\times\mathbb{S}_2^n$ with second fundamental form $B$ and normal connection $\nabla'.$ It appears that they are also sufficient:

\begin{thm} \label{fal thm SmSn}
Let us assume that $B:TM \times TM \to E$ is symmetric and such that the Gauss, Ricci and Codazzi Equations (\ref{Gauss equation}), (\ref{Ricci equation}) and (\ref{Codazzi equation}) hold together with \eqref{eqn fhst 1}-\eqref{eqn fhst 4}. Let us moreover suppose that $\dim\mbox{Ker}(\mathcal{P}-id)=m$ and $\dim\mbox{Ker}(\mathcal{P}+id)=n.$ Then there exists $\varphi \in \Gamma(U \Sigma)$ solution of \eqref{killing equation SmSn} such that the map
$$\Phi:\hspace{.3cm}TM \oplus E \to \R^{m+n+2},\hspace{.3cm}Z \mapsto \langle\langle Z \cdot \varphi, \varphi \rangle\rangle$$
commutes with the product structures $\mathcal{P}$ and $\mathcal{P}'$. The spinor field $\varphi$ is moreover unique up to the natural right action of $Spin(m+1)\cdot Spin(n+1)$ on $U\Sigma.$ In parti\-cular, there is an isometric immersion $F:M\rightarrow \mathbb{S}_1^m\times  \mathbb{S}_2^n$ and a bundle isomorphism $\Phi:TM\oplus E\rightarrow T(\mathbb{S}_1^m\times  \mathbb{S}_2^n)$ above $F$ identifying $E,B$ and $\nabla'$ to the normal bundle, the second fundamental form and the normal connection of $F$ in $\mathbb{S}_1^m\times  \mathbb{S}_2^n.$ The immersion is moreover unique up to the natural action of $SO(m+1)\times SO(n+1)$ on $\mathbb{S}_1^m\times  \mathbb{S}_2^n.$ 
\end{thm}
Section \ref{sub sec proof prop} is devoted to the proof of Proposition \ref{prop killing implies fundamental equations}, and Section \ref{sub sec proof thm} to the proof of Theorem \ref{fal thm SmSn}.
\subsubsection{Proof of Proposition \ref{prop killing implies fundamental equations}}\label{sub sec proof prop}
We assume that $X,Y\in \Gamma(TM)$ are such that $\nabla X=\nabla Y=0$ at the point where we do the computations. A direct computation using (\ref{killing equation SmSn}) twice yields
\begin{eqnarray}
R_{XY} \varphi &=&\nabla^2_{X,Y}\varphi-\nabla^2_{Y,X}\varphi\nonumber\\
&=& \underbrace{\frac{1}{2} \bigg( (\nabla_Y B)(X) - (\nabla_X B)(Y) \bigg)}_\mathcal{A} \cdot \varphi  + \underbrace{\frac{1}{4} \bigg( B(Y) \cdot B(X) - B(X) \cdot B(Y) \bigg)}_\mathcal{B}  \cdot \varphi\nonumber\\
& &+ \underbrace{ \frac{1}{4}\bigg( c_1(Y_1 \cdot X_1 - X_1 \cdot Y_1) +c_2(Y_2 \cdot X_2- X_2 \cdot Y_2) \bigg)}_\mathcal{C}  \cdot \varphi+ \left(\mathcal{D} + \mathcal{E}\right)\cdot\varphi\label{R1 pf fal SmSn}
\end{eqnarray}
with
\begin{eqnarray*}
\mathcal{D}&=&\frac{\sqrt{c_1}}{4} \bigg(B(X)\cdot Y_1 \cdot \nu_1-Y_1 \cdot \nu_1 \cdot B(X) -B(Y)\cdot X_1 \cdot \nu_1+X_1 \cdot \nu_1 \cdot B(Y)\bigg) \\ 
&&+ \frac{\sqrt{c_2}}{4} \bigg( B(X)\cdot Y_2 \cdot \nu_2 -Y_2 \cdot \nu_2 \cdot B(X)- B(Y)\cdot X_2 \cdot \nu_2 +X_2 \cdot \nu_2 \cdot B(Y)  \bigg)
\end{eqnarray*} 
and 
 $$\mathcal{E}= \frac{1}{2} \bigg( \sqrt{c_1} (\nabla_X Y_1 - \nabla_YX_1) \cdot \nu_1 +\sqrt{c_2}(\nabla_XY_2 - \nabla_YX_2)\cdot \nu_2 \bigg).$$
\begin{lem} \label{lem:1} In local orthonormal frames  $\{e_j\}_{1\leq j \leq p}$ of $TM$ and $\{n_r\}_{1 \leq r \leq q}$ of $E$ we have
\begin{equation}\label{lem expr A}
\mathcal{A}= \frac{1}{2} \sum_{j=1}^p e_j \cdot \big( (\widetilde{\nabla}_YB)(X,e_j)-(\widetilde{\nabla}_XB)(Y,e_j) \big),
\end{equation}
\begin{eqnarray}
\ \ \ \mathcal{B} &=& \frac{1}{2} \sum_{1\leq j < k\leq p} \big( \langle B^*(X,B(Y,e_j)), e_k \rangle - \langle B^*(Y,B(X,e_j)),e_k \rangle \big) e_j  \cdot e_ k \label{lem expr B}  \\
&&+  \frac{1}{2}  \sum_{1\leq r < s\leq q}\big(  \langle B(X,B^*(Y,n_r)), n_s \rangle - \langle B(Y,B^*(X,n_r)),n_s \rangle  \big) n_r \cdot n_s,\nonumber 
\end{eqnarray}
\begin{eqnarray}
\mathcal{C} &=&  -\frac{1}{8}\left(c_1+c_2\right)\sum_{1 \leq j < k \leq p} \langle ( X \wedge Y + f(X) \wedge f(Y) ) e_j , e_k \rangle e_j \cdot e_k \label{lem expr C} \\
&& -\frac{1}{8}\left(c_1+c_2\right)\sum_{1 \leq r < s \leq q} \langle  (h(X) \wedge h(Y))n_r,n_s \rangle n_r \cdot n_s\nonumber\\
&&+\frac{1}{8}\left(c_1+c_2)\right)\sum_{j=1}^p e_j \cdot ( \langle f(Y) , e_j \rangle h(X) - \langle f(X) , e_j \rangle h(Y))\nonumber\\
 && -\frac{1}{8}\left(c_1-c_2\right)\sum_{1 \leq j < k \leq p} \langle ( f(X) \wedge Y + X \wedge f(Y) ) e_j , e_k \rangle e_j \cdot e_k \nonumber\\
 &&+\frac{1}{8}\left(c_1-c_2\right)\sum_{j=1}^p e_j \cdot ( \langle Y , e_j \rangle h(X) - \langle X , e_j \rangle h(Y)).\nonumber
 \end{eqnarray}
These expressions respectively mean that $\mathcal{A}\in TM\otimes E$ represents the transformation 
\begin{equation}\label{endo A}
Z\in TM\mapsto \widetilde{\nabla}_YB(X,Z)- \widetilde{\nabla}_XB(Y,Z)\ \in E,
\end{equation}
$\mathcal{B}\in \Lambda^2TM\oplus \Lambda^2E$ represents the transformation
\begin{equation}\label{endo B1}
Z\in TM\mapsto B^*(X,B(Y,Z))-B^*(Y,B(X,Z))\ \in TM
\end{equation}
together with
\begin{equation}\label{endo B2}
N\in E\mapsto B(X,B^*(Y,N))-B(Y,B^*(X,N))\ \in E
\end{equation}
and $\mathcal{C}\in\ \Lambda^2TM\ \oplus\ \Lambda^2E\ \oplus\ TM\otimes E$ represents 
\begin{equation}\label{interp C1}
-\frac{1}{4}\left(c_1+c_2\right)(X\wedge Y + f(X)\wedge f(Y))\ \in End(TM),
\end{equation}
\begin{equation}\label{interp C2}
-\frac{1}{4}(c_1+c_2)(h(X)\wedge h(Y))\ \in End(E),
\end{equation}
\begin{equation}\label{interp C3}
Z\in TM\mapsto \frac{1}{4}(c_1+c_2)\left(\langle f(Y) , Z \rangle h(X) - \langle f(X) , Z \rangle h(Y)\right)\in E,
\end{equation}
\begin{equation}\label{interp C4}
-\frac{1}{4}(c_1-c_2)\left(f(X)\wedge Y+X\wedge f(Y)\right)\ \in End(TM)
\end{equation}
and
\begin{equation}\label{interp C5}
Z\in TM\mapsto \frac{1}{4}(c_1-c_2)\left(\langle Y , Z \rangle h(X) - \langle X , Z \rangle h(Y)\right)\in E.
\end{equation}
\end{lem}
\begin{proof}
The expression (\ref{lem expr A}) directly follows from the definition (\ref{def B Clifford}) of $B(X).$ For (\ref{lem expr B}) we refer to \cite[Lemma 5.2]{BLR2} where a similar computation is carried out. By Lemma \ref{lem1 ap1}, formula (\ref{biv rep u2}), $\mathcal{A}$ represents the transformation (\ref{endo A}) and $\mathcal{B}$ the transformations (\ref{endo B1}) and (\ref{endo B2}). We now prove (\ref{lem expr C}). Using
$$X_1=\frac{1}{2}\left(X+f(X)+h(X)\right)\hspace{.5cm}\mbox{and}\hspace{.5cm}X_2=\frac{1}{2}\left(X-f(X)-h(X)\right)$$
and the analogous expressions for $Y,$ straightforward computations yield
\begin{eqnarray}
\mathcal{C}&=&\frac{1}{16}(c_1+c_2)\left\{\left(Y\cdot X-X\cdot Y\right)+\left(f(Y)\cdot f(X)-f(X)\cdot f(Y)\right)\right\}\label{dem lem expr C1}\\
&&+\frac{1}{16}(c_1+c_2)\left\{\left(h(Y)\cdot h(X)-h(X)\cdot h(Y)\right)\right\}\label{dem lem expr C2}\\
&&+\frac{1}{8}(c_1+c_2)\left(f(Y)\cdot h(X)-f(X)\cdot h(Y)\right)\label{dem lem expr C3}\\
&&+\frac{1}{16}(c_1-c_2)\left\{Y\cdot f(X)-f(X)\cdot Y-X\cdot f(Y)+f(Y)\cdot X\right\}\label{dem lem expr C4}\\
&&+\frac{1}{8}(c_1-c_2)\left\{Y\cdot h(X)-X\cdot h(Y)\right\}.\label{dem lem expr C5}
\end{eqnarray}
By Lemma \ref{lem rep u wedge v} in the Appendix, the right hand terms (\ref{dem lem expr C1}) and (\ref{dem lem expr C2}) represent the transformations (\ref{interp C1}) and (\ref{interp C2}); the term (\ref{dem lem expr C3}) represents the transformation (\ref{interp C3}) since the commutator in the Clifford bundle
$$\alpha:=\left[\frac{1}{4}(f(Y)\cdot h(X)-f(X)\cdot h(Y)),Z\right]$$
is
\begin{eqnarray*}
\alpha&=&\frac{1}{4}(f(Y)\cdot h(X)-f(X)\cdot h(Y))\cdot Z-Z\cdot\frac{1}{4}(f(Y)\cdot h(X)-f(X)\cdot h(Y))\\
&=&-\frac{1}{4}(f(Y)\cdot Z+Z\cdot f(Y))\cdot h(X)+\frac{1}{4}(f(X)\cdot Z+Z\cdot f(X))\cdot h(Y)\\
&=&\frac{1}{2}\left(\langle f(Y),Z\rangle h(X)-\langle f(X),Z\rangle h(Y)\right);
\end{eqnarray*}
similarly, the terms (\ref{dem lem expr C4}) and (\ref{dem lem expr C5}) represent the transformations (\ref{interp C4}) and (\ref{interp C5}). Formula (\ref{lem expr C}) then follows from Lemmas \ref{lem1 ap1} and \ref{lem3 ap1} in the Appendix.
\end{proof}
The curvature tensor of the spinorial connection on $TM\oplus E$ is given by
\begin{eqnarray}
R_{XY}  \varphi &=& \bigg( \frac{1}{2}  \sum_{1\leq  j < k \leq p } \langle R^T(X,Y) e_j , e_k \rangle e_j \cdot e_k \bigg) \cdot \varphi \label{R2 pf fal SmSn}\\ 
&& \ \ \ \ + \bigg( \frac{1}{2}\sum_{1 \leq r < s \leq q } \langle R^N(X,Y) n_r , n_s \rangle n_r \cdot n_s \bigg) \cdot \varphi.\nonumber
\end{eqnarray}
Comparing Equations (\ref{R1 pf fal SmSn}) and (\ref{R2 pf fal SmSn}) and since $\varphi$ is represented in a frame $\widetilde{s}\in\widetilde{Q}$ by an element of $Spin(m+n+2),$ invertible in $Cl(m+n+2),$ we deduce that
\begin{eqnarray}
\mathcal{A}+\mathcal{B} + \mathcal{C}&=&  \frac{1}{2}  \sum_{1\leq  j < k \leq p } \langle R^T(X,Y) e_j , e_k \rangle e_j \cdot e_k\label{ABC R}\\
&&+\frac{1}{2}\sum_{1 \leq r < s \leq q } \langle R^N(X,Y) n_r , n_s \rangle n_r \cdot n_s\nonumber
\end{eqnarray}
and 
\begin{equation}\label{DE R}
\mathcal{D}+\mathcal{E}=0.
\end{equation} 
Now the right hand side of (\ref{ABC R}) represents the transformations $Z\in TM\mapsto R^T(X,Y)Z\in TM$ and $N\in E\mapsto R^N(X,Y)N\in E.$ The equations (\ref{Gauss equation})-(\ref{Codazzi equation}) of Gauss, Ricci and Codazzi follow from this and Lemma \ref{lem:1}. Let us now prove that Equations (\ref{eqn fhst 1})-(\ref{eqn fhst 4}) are consequences of the fact that $\Phi$ commutes with the product structures $\mathcal{P}$ and $\mathcal{P}',$ Equation (\ref{eqn lem app th fal}), where the product structure $\mathcal{P}'$ is parallel. We have by (\ref{eqn lem app th fal})
\begin{equation}\label{der eqn lem app th fal}
\mathcal{P}'\left(\partial_Y\Phi(X)\right)=\partial_Y\left(\Phi\left(\mathcal{P}(X)\right)\right).
\end{equation}
Assuming that $\nabla X=0$ at the point where we do the computations and recallling (\ref{partial X phi dem thm}) we have
\begin{equation*}
\partial_Y\Phi(X)=-\sum_{i=1,2}\sqrt{c_i}\ \langle Y_i,X\rangle N_i+\Phi(B(Y,X_T)-B^*(Y,X_N))
\end{equation*}
and the left hand side of (\ref{der eqn lem app th fal}) is given by
\begin{equation}\label{expr pYPpPhiX2}
\mathcal{P}'\left(\partial_Y\Phi(X)\right)=-\sum_{i=1,2}\sqrt{c_i}\ \langle Y_i,X\rangle \mathcal{P}'(N_i)+\Phi(\mathcal{P}(B(Y,X_T))-\mathcal{P}(B^*(Y,X_N)));
\end{equation}
by (\ref{partial X phi dem thm}) again, the right hand side of (\ref{der eqn lem app th fal}) is given by
\begin{equation}\label{dem lem th fal expr RHST}
\begin{split}
\partial_Y\left(\Phi(\mathcal{P}(X))\right)=\Phi(\nabla_Y\mathcal{P}(X))-\sum_{i=1,2}\sqrt{c_i}\ \langle Y_i,\mathcal{P}(X)\rangle N_i\\+\Phi(B(Y,\mathcal{P}(X)_T)-B^*(Y,\mathcal{P}(X)_N))
\end{split}
\end{equation}
and since for $i=1,2$
$$\langle Y_i,\mathcal{P}(X)\rangle N_i=\langle \mathcal{P}(Y_i),X\rangle N_i=(-1)^{i+1}\langle Y_i,X\rangle N_i=\langle Y_i,X\rangle\mathcal{P}'(N_i)$$
we deduce that for $X\in TM$
$$\Phi ( \mathcal{P}(B(Y,X)) )=\Phi(\nabla_Y \mathcal{P}(X)) + \Phi \left( B( Y , f(X) )  - B^* ( Y  , h(X)) \right)$$
and for $X \in E$
$$- \Phi ( \mathcal{P}(B^*(Y,X))) =  \Phi(\nabla_Y \mathcal{P}(X)) + \Phi \left( B (Y,s(X)) - B^*(Y,t(X)) \right).$$
Using that $\Phi$ is injective on the fibers and decomposing $\nabla_Y \mathcal{P}(X)$, $\mathcal{P}(B(Y,X))$ and $\mathcal{P}(B^*(Y,X))$ in their tangent and normal parts, we get
$$(\widetilde{\nabla}_Y f)(X) + (\widetilde{\nabla}_Y h)(X)= s(B(Y,X)) + t (B(Y,X))-B(Y,f(X)) + B^*(Y,h(X))$$  if $X \in TM,$ and
$$(\widetilde{\nabla}_Y s)(X) + (\widetilde{\nabla}_Y t)(X) = - f (B^*(Y,X)) - h (B^*(Y,X)) - B(Y,s(X)) - B^*(Y,t(X))$$ if $X \in E.$ 
Finally, taking the tangent and the normal parts of each one of the last two equations we get (\ref{eqn fhst 1})-(\ref{eqn fhst 4}). 
\begin{rem}
Eq. (\ref{DE R}) is in fact equivalent to the antisymmetric part of (\ref{eqn fhst 1})-(\ref{eqn fhst 4}).
\end{rem}

\subsubsection{Proof of Theorem \ref{fal thm SmSn}} \label{sub sec proof thm}
Let us set, for $X\in TM$ and $\varphi\in \Gamma(U\Sigma),$
\begin{equation*}
\nabla'_X \varphi := \nabla_X \varphi  -\frac{1}{2} \left(\sqrt{c_1}\ X_1 \cdot \nu_1 + \sqrt{c_2}\ X_2 \cdot \nu_2-B(X)\right) \cdot \varphi.
\end{equation*}
We consider $U\Sigma\rightarrow M$ as a principal bundle of group $Spin(p+q+2),$ where the action is the multiplication on the right
$$\varphi=\left[\widetilde{s},[\varphi]\right]\mapsto \varphi\cdot a:=\left[\widetilde{s},[\varphi]\cdot a\right]$$
for all $a\in Spin(p+q+2).$ The connection $\nabla'$ may be considered as given by a connection 1-form on this principal bundle, since so is $\nabla$ and the term
$$\mathcal{X}(\varphi):= \frac{1}{2} \left(\sqrt{c_1}\ X_1 \cdot \nu_1 + \sqrt{c_2}\ X_2 \cdot \nu_2-B(X)\right) \cdot \varphi$$ 
defines a vertical and invariant vector field on $U\Sigma.$ The compatibility equations (\ref{Gauss equation})-(\ref{eqn fhst 4}) imply that this connection is flat (the computations are similar to the computations in the previous section). Since it is flat  and assuming moreover that $M$ is simply connected, the principal bundle $U\Sigma\rightarrow M$ has a global parallel section: this yields $\varphi\in\Gamma(U\Sigma)$ such that $\nabla'\varphi=0,$ i.e. a non-trivial solution of (\ref{killing equation SmSn}). Let us verify that equations (\ref{eqn fhst 1})-(\ref{eqn fhst 4}) imply that the map 
$$\Phi:\hspace{.3cm}TM \oplus E \to \R^{m+n+2}=\R^{m+1}\times\R^{n+1},\hspace{.3cm}X \mapsto \langle\langle X \cdot \varphi , \varphi\rangle\rangle$$
is compatible with the product structures, i.e. verifies $\Phi (\mathcal{P}(X)) = \mathcal{P}' (\Phi (X))$ for all $X \in TM \oplus E.$
The sum of (\ref{eqn fhst 1}) and (\ref{eqn fhst 2}) gives, for $X,Y\in TM,$
\begin{equation}\label{nabla P B B*}
\nabla_Y \mathcal{P}(X) =\mathcal{P}(B(Y,X)) - B(Y,f(X)) + B^*(Y,h(X)).
\end{equation}
Similarly, for $X\in E$ and $Y\in TM,$ (\ref{eqn fhst 3}) and (\ref{eqn fhst 4}) imply that 
\begin{equation}\label{nabla P B s t}
\nabla_Y \mathcal{P}(X)= - \mathcal{P}(B^*(Y,X)) - B(Y,s(X)) + B^*(Y,t(X)).
\end{equation}
As in the proof of Theorem \ref{th main result SmSn}, $\Phi$ is a bundle map above the immersion 
\begin{eqnarray*}
F:\ M &\to& \R^{m+1} \times \R^{n+1} \\
p &\mapsto& (\frac{1}{\sqrt{c_1}}\langle\langle \nu_1 \cdot \varphi , \varphi \rangle\rangle,\frac{1}{\sqrt{c_2}}\langle\langle\nu_2\cdot \varphi , \varphi \rangle\rangle).
\end{eqnarray*}
The product structure $\mathcal{P}$ on $TM \oplus E$ extends to a product structure $\widetilde{\mathcal{P}}$ on $TM \oplus E \oplus \mathcal{E}_2$ by setting $\widetilde{\mathcal{P}}(\nu_1)=\nu_1$ and $\widetilde{\mathcal{P}}(\nu_2)=-\nu_2.$ Let us consider the trivial connection $\partial$ induced on $TM\oplus E\oplus \mathcal{E}_2$ by the bundle isomorphism
$$\widetilde{\Phi}:\ TM\oplus E\oplus \mathcal{E}_2\rightarrow F^*T\R^{m+n+2}.$$ 
\begin{lem}
We have for $Y\in\Gamma(TM\oplus E)$ and $X\in TM$
\begin{equation}\label{del moins nabla}
(\partial-\nabla)_XY= -\sum_{i=1,2}\sqrt{c_i}\ \langle X_i , Y \rangle \nu_i +B(X,Y_T)-B^*(X,Y_N)
\end{equation}
and $(\partial-\nabla)_X\nu_i=X_i,$ $i=1,2.$
\end{lem}
\begin{proof}
Assuming that $\nabla_X Y = 0$ at the point where we do the computations, we have by definition $(\partial- \nabla)_XY=\widetilde \Phi^{-1} (\partial_X \Phi (Y) ),$ and the formula is a consequence of (\ref{partial X phi dem thm}). Finally, $(\partial-\nabla)_X\nu_i=\partial_X\nu_i=\widetilde \Phi^{-1} (\partial_X N_i)=X_i.$
\end{proof}
\begin{lem}
The product structure $\widetilde{\mathcal{P}}$ is parallel with respect to $\partial$.
\end{lem}
\begin{proof}
Using (\ref{del moins nabla}) twice, for $X,Y$ tangent to $M$ we have 
\begin{eqnarray*}
((\partial-\nabla)_X \widetilde{\mathcal{P}})(Y)&=& (\partial - \nabla)_X (\widetilde{\mathcal{P}} (Y)) -\widetilde{\mathcal{P}}((\partial - \nabla)_XY)\\
&=& B(X,f(Y))-B^*(X,h(Y)) -\widetilde{\mathcal{P}}(B(X,Y))
\end{eqnarray*}
since $\langle Y_1 ,  \widetilde{\mathcal{P}}(X) \rangle = \langle \widetilde{\mathcal{P}}(Y_1) , \widetilde{\mathcal{P}}(X) \rangle = \langle Y_1 , X \rangle$
and $\langle Y_2 , \widetilde{\mathcal{P}} (X) \rangle = - \langle \widetilde{\mathcal{P}} (Y_2) , \widetilde{\mathcal{P}}(X) \rangle = - \langle Y_2 , X \rangle,$ and we conclude with (\ref{nabla P B B*}) that $(\partial_X \widetilde{\mathcal{P}})(Y) = 0.$ The computation for $Y\in\Gamma(E)$ is analogous. For $Y=\nu_1$ we have 
\begin{eqnarray*}
((\partial-\nabla)_X \widetilde{\mathcal{P}})(\nu_1)&=& (\partial - \nabla)_X (\widetilde{\mathcal{P}} (\nu_1)) -\widetilde{\mathcal{P}}((\partial - \nabla)_X\nu_1)\\
&=&(\partial - \nabla)_X \nu_1 - \widetilde{\mathcal{P}}(X_1)\\
&=& X_1 - \mathcal{P} (X_1) = 0
\end{eqnarray*}
which implies that $(\partial_X \widetilde{\mathcal{P}})(\nu_1) = 0$ since
$$(\nabla_X \widetilde{\mathcal{P}})(\nu_1)=\nabla_X (\widetilde{\mathcal{P}}(\nu_1))-\widetilde{\mathcal{P}} (\nabla_X\nu_1)=\nabla_X\nu_1-\widetilde{\mathcal{P}}(\nabla_X\nu_1)=0.$$
The computations for $X=\nu_2$ are analogous.
\end{proof}
Since  $\widetilde{\mathcal{P}}$ and $\mathcal{P}'_{\vert M}$ are parallel sections of endomorphisms of 
$$TM \oplus E \oplus \mathcal{E}_2\cong F^*(T\R^{m+n+2})$$ 
and since $id +\widetilde{\mathcal{P}}$ and $id- \widetilde{\mathcal{P}}$ have rank $m+1$ and $n+1,$ there exists $A \in O(m+n+2)$ such that
 \begin{align*}
A \circ \widetilde\Phi \circ \widetilde{\mathcal{P}} \circ \widetilde \Phi^{-1}  \circ A^{-1}= \mathcal{P}' 
 \end{align*} 
on $F^*T\R^{n+m+2}$. We consider $a \in Spin(m+n+2)$ such that $Ad(a)= A^{-1}$ and the spinor field $\varphi' := \varphi \cdot a \in U\Sigma:$ it is still a solution of (\ref{killing equation SmSn}) and $\widetilde{\Phi}'(X):=\langle\langle X \cdot \varphi' , \varphi' \rangle\rangle$ is such that
\begin{eqnarray*}
\widetilde{\Phi}' (X) &=& \tau [\varphi'][X][\varphi'] = \tau[\varphi \cdot a][X][ \varphi \cdot a]\\& =& a^{-1} [\varphi]^{-1} [X] [\varphi] a= Ad(a^{-1})(\widetilde{\Phi}(X))= A\circ\widetilde{\Phi}(X).
\end{eqnarray*} 
The map $\widetilde{\Phi}'$ thus satisfies $\widetilde\Phi' \circ \widetilde{\mathcal{P}} \circ \widetilde \Phi'^{-1}= \mathcal{P}'$ which implies that $\Phi':TM\oplus E\rightarrow F^*T(\mathbb{S}_1^m\times\mathbb{S}_2^n)$ is compatible with the product structures $\mathcal{P}$ and $\mathcal{P}'.$ Finally, it is clear from the proof that if a solution $\varphi$ of (\ref{killing equation SmSn}) is such that $\Phi:X\mapsto \langle\langle X\cdot\varphi,\varphi\rangle\rangle$ commutes with the product structures, then the other solutions of (\ref{killing equation SmSn}) satisfying this property are of the form $\varphi\cdot a$ with $a\in Spin(m+n+2)$ such that $Ad(a)$ belongs to $SO(m+1)\times SO(n+1),$ i.e. with $a\in Spin(m+1)\cdot Spin(n+1).$

\section{Isometric immersions in $\mathbb{S}_1^m\times \mathbb{R}^n$}\label{section SmRn}
We now consider immersions in $\mathbb{S}_1^m\times \mathbb{R}^n$ where $\mathbb{S}_1^m$ is a $m$-dimensional sphere of curvature $c_1>0$. After the statement of the main theorem in Section \ref{section main thm SmRn}, we study the special cases $\mathbb{S}^2\times\R$ and $\mathbb{S}^2\times\R^2$ in Sections \ref{section S2R} and \ref{section S3R S2R2}. 

In that section $M$ still denotes a $p$-dimensional riemannian manifold and $E\rightarrow M$ a metric bundle of rank $q$ with $p+q=m+n,$ equipped with a connection compatible with the metric. We consider here the trivial bundle $\mathcal{E}_1:=M\times\R\rightarrow M,$ with its natural metric and the trivial connection, and fix a unit parallel section $\nu_1$ of $\mathcal{E}_1.$ We finally consider the representation associated to the splitting $\R^{m+n+1}=\R^p\oplus\R^q\oplus\R$
$$\rho:\ Spin(p)\times Spin(q)\rightarrow Spin(p)\cdot Spin(q)\ \subset Spin(m+n+1)\rightarrow Aut(Cl(m+n+1))$$
(the last map is given by the left multiplication) and the bundles (associated to a spin structure $\widetilde{Q}:=\widetilde{Q}_M\times_M\widetilde{Q}_E$ of $TM$ and $E$)
$$\Sigma:=\widetilde{Q}\times_{\rho}Cl(m+n+1),\hspace{1cm} U\Sigma:=\widetilde{Q}\times_{\rho}Spin(m+n+1)$$
and
$$Cl(TM\oplus E\oplus\mathcal{E}_1):=\widetilde{Q}\times_{Ad}Cl(m+n+1).$$
We finally suppose that a product structure $\mathcal{P}$ is given on $TM\oplus E$ as in Section \ref{section main thm SmSn}.
\subsection{Statement of the theorem}\label{section main thm SmRn}
\begin{thm}\label{th main result SmRn}
We suppose that $M$ is simply connected. Let $B: TM \times TM \to E$ be a symmetric tensor. The following statements are equivalent:
\begin{enumerate}
\item[(i)] There exist an isometric immersion $F:M\rightarrow \mathbb{S}_1^m\times \mathbb{R}^n$ and a bundle map $\Phi: TM \oplus E \to T\mathbb{S}_1^m \times \mathbb{R}^n$ above $F$ such that $\Phi(X,0)= dF(X)$ for all $X\in TM,$ which preserves the bundle metrics, maps the connection on $E$ and the tensor $B$ to the normal connection and the second fundamental form of $F$, and  is compatible with the product structures.
\\
\item[(ii)] There exists a section $\varphi \in \Gamma(U \Sigma )$ solution of
\begin{equation}\label{killing eqn SmRn}
\nabla_X \varphi = \frac{1}{2} \sqrt{c_1}\ X_1 \cdot \nu_1\cdot \varphi  -  \frac{1}{2}B(X)\cdot \varphi
\end{equation}
for all $X\in TM,$ where $X= X_1 + X_2$ is the decomposition in the product structure $\mathcal{P}$ of $TM\oplus E$, such that the map
$$Z\in TM\oplus E\ \mapsto\ \langle\langle Z \cdot \varphi , \varphi \rangle\rangle\in\R^{m+1}\times\R^{n}$$
commutes with the product structures $\mathcal{P}$ and $\mathcal{P}'.$
\end{enumerate}
Moreover, the bundle map $\Phi$ and the immersion $F$ are explicitly given in terms of the spinor field $\varphi$ by the formulas
$$\Phi:\hspace{.3cm}TM \oplus E \to T\mathbb{S}_1^m \times \mathbb{R}^n,\hspace{.3cm}Z \mapsto \langle\langle Z \cdot \varphi , \varphi \rangle\rangle$$
and $F=(F_1,F_2)\in\mathbb{S}_1^m\times\R^n$ with
\begin{equation}\label{explicit f SmRn}
F_1= \frac{1}{\sqrt{c_1}}\langle\langle \nu_1 \cdot \varphi ,  \varphi \rangle\rangle.
\end{equation}
\end{thm}
\noindent\textit{Brief indications of the proof:} setting
\begin{equation}\label{explicit f 2 SmRn}
\Phi_2(X)=\langle\langle X_2\cdot\varphi,\varphi\rangle\rangle\hspace{.5cm}\mbox{and}\hspace{.5cm}F_2=\int\Phi_2
\end{equation}
and using (\ref{eqn fhst 1})-(\ref{eqn fhst 4}) it is not difficult to see that $\Phi_2$ is a closed 1-form and $F_2$ is well defined if $M$ is simply connected. Formulas (\ref{explicit f SmRn}) and (\ref{explicit f 2 SmRn}) thus give an explicit expression for $F=(F_1,F_2)$ in terms of the spinor field $\varphi,$ and the theorem may then be proved by direct computations as in the previous sections.
\\

Here again, as in the case of a product of two spheres, we can obtain a spinorial proof of the fundamental theorem of immersions theory in $\mathbb{S}_1^m\times\mathbb{R}^n.$ 

\subsection{Surfaces in $\mathbb{S}^2\times\R$}\label{section S2R}
The aim is to recover the spinorial characterization of an immersion in $\mathbb{S}^2\times\R$ given in \cite{Ro1}. Let us consider $\Sigma_0=\widetilde{Q}\times_{\rho}Cl^0(4).$ If $e_0^o,e_1^o,e_2^o,e_3^o$ is an orthonormal basis of $\R^4,$ where $e_0^o$ belongs to the second factor of $\mathbb{S}^2\times\R,$ we set $\omega:=-e_0^o\cdot e_1^o\cdot e_2^o\cdot e_3^o,$ consider the two ideals $\mathcal{I}_1:=Cl^0(4)\cdot \frac{1}{2}\left(1-\omega\right)$ and $\mathcal{I}_2:=Cl^0(4)\cdot \frac{1}{2}\left(1+\omega\right)$ of $Cl^0(4)$
and the splitting $Cl^0(4)=\mathcal{I}_1\oplus \mathcal{I}_2.$ It induces a decomposition
\begin{equation}\label{dec1 S2R}
\varphi=\varphi_1+\varphi_2\hspace{.3cm}\in\ \Sigma_1\oplus\Sigma_2
\end{equation}
with $\Sigma_1=\widetilde{Q}\times_{\rho}\mathcal{I}_1$ and $\Sigma_2=\widetilde{Q}\times_{\rho}\mathcal{I}_2.$ Let us consider the map
$$u:\hspace{.3cm}\Sigma_2\rightarrow\Sigma_1, \hspace{.3cm}\varphi_2\mapsto u(\varphi_2)=-\nu_1\cdot\varphi_2\cdot e_0^o$$
to identify $\Sigma_2$ with $\Sigma_1,$ and an identification
$$\Sigma M\otimes\Sigma E\rightarrow\Sigma_1,\hspace{.3cm}\psi\mapsto\psi^*$$
such that $(X\cdot\psi)^*=X\cdot\nu_1\cdot(\psi)^*$ for all $X\in TM\oplus E$ and $\psi\in \Sigma M\otimes\Sigma E.$ We set $\psi_1,\psi_2\in \Sigma M\otimes\Sigma E$ such that $\psi_1^*=\varphi_1$ and $\psi_2^*=u(\varphi_2).$ Since $\varphi_1$ and $\varphi_2$ are both normalized solutions of
$$\nabla_X\varphi=\frac{1}{2}X_1\cdot\nu_1\cdot\varphi-\frac{1}{2}S(X)\cdot N\cdot\varphi$$
where $N$ is a unit normal and $S:TM\rightarrow TM$ is the corresponding shape operator of $M$ in $\mathbb{S}^2\times\R,$ $\psi_1$ and $\psi_2\in \Sigma M\otimes\Sigma E$ are so that
\begin{equation}\label{eqn psi1 S2R}
\nabla_X\psi_1=\frac{1}{2}X_1\cdot\psi_1-\frac{1}{2}S(X)\cdot N\cdot\psi_1
\end{equation}
and
\begin{equation}\label{eqn psi2 S2R}
\nabla_X\psi_2=-\frac{1}{2}X_1\cdot\psi_2-\frac{1}{2}S(X)\cdot N\cdot\psi_2
\end{equation}
with $|\psi_1|=|\psi_2|=1.$ Now the condition expressing that $\Phi$ commutes with the product structures gives the following: 
\begin{lem}\label{lem trad str prod pres S2R}
For a convenient choice of the unit section $V\in \Gamma(TM\oplus E)$ generating the distinguished line $\mathcal{P}_2$ of the product structure $\mathcal{P}$ of $TM\oplus E$, we have
\begin{equation}\label{trad str prod pres S2R}
V\cdot\psi_1=\psi_2.
\end{equation}
\end{lem}
\begin{proof}
Choosing $V\in \mathcal{P}_2$ so that $\Phi(V)=e_0^o,$ we have
$$\Phi(V)=\langle\langle V\cdot\varphi,\varphi\rangle\rangle=\tau[\varphi][V][\varphi]=e_0^o,$$
that is $[V][\varphi]=[\varphi]e_0^o.$ Writing $[\varphi]=[\varphi_1]+[\varphi_2]\in \mathcal{I}_1\oplus\mathcal{I}_2$ and since the right-multiplication by $e_0^o$ exchanges the ideals $\mathcal{I}_1$ and $\mathcal{I}_2$  (since $\omega\cdot e_0^o=-e_0^o\cdot \omega$), we deduce that $[V][\varphi_1]=[\varphi_2]e_0^o$ and $[V][\varphi_2]=[\varphi_1]e_0^o.$ We thus have $[V][\nu_1][\varphi_1]=-[\nu_1][\varphi_2]e_0^o$ that is $V\cdot \nu_1\cdot\varphi_1=-\nu_1\cdot\varphi_2\cdot e_0^o,$ which readily implies (\ref{trad str prod pres S2R}).
\end{proof}
Equations (\ref{eqn psi1 S2R}) and (\ref{eqn psi2 S2R}) and the lemma imply that $\psi_1$ and $\psi_2$ satisfy 
$$\nabla_X\psi_1=-\frac{1}{2}X_1\cdot V\cdot\psi_2-\frac{1}{2}S(X)\cdot N\cdot\psi_1.$$
and
$$\nabla_X\psi_2=-\frac{1}{2}X_1\cdot V\cdot\psi_1-\frac{1}{2}S(X)\cdot N\cdot\psi_2.$$
The spinor field $\psi:=\psi_1-\psi_2\in\Sigma M\otimes\Sigma E$ is a solution of
\begin{equation}\label{killing S2R 1}
\nabla_X\psi=\frac{1}{2}X_1\cdot V\cdot\psi-\frac{1}{2}S(X)\cdot N\cdot\psi.
\end{equation}
By (\ref{trad str prod pres S2R}), we have $\langle\psi_1,\psi_2\rangle= \langle\psi_1,V\cdot \psi_1\rangle=- \langle V\cdot\psi_1,\psi_1\rangle=- \langle \psi_2,\psi_1\rangle=0,$
that is $\psi_1$ and $\psi_2$ are orthogonal in $\Sigma M\otimes\Sigma E,$ which implies that $|\psi|=\sqrt{2}.$ Finally, since $E=\R N$ there is an identification
$$\Sigma M\rightarrow\Sigma M\otimes\Sigma E,\hspace{.3cm}\psi\mapsto\psi^*$$
such that $(X\cdot\psi)^*=X\cdot N\cdot(\psi)^*.$ Using that $X_1=X-\langle X,V\rangle V$ and $V=T+fN,$ we readily get from (\ref{killing S2R 1}) that
$$\nabla_X\psi=\frac{1}{2}X\cdot T\cdot\psi+\frac{1}{2}fX\cdot\psi+\frac{1}{2}\langle X,T\rangle\psi-\frac{1}{2}S(X)\cdot\psi.$$
This is the spinorial characterization of an immersion in $\mathbb{S}^2\times\R$ obtained in \cite{Ro1}. 
\begin{rem}
Similarly, it is possible to obtain as a consequence of Theorem \ref{th main result SmRn} the characterizations in terms of usual spinor fields of immersions of surfaces or hypersurfaces in $\mathbb{S}^3\times\R,$ or of surfaces in $\mathbb{S}^2\times\R^2,$ obtained in \cite{LR,Ro2}. We rather focus below on the new case of hypersurfaces in $\mathbb{S}^2\times\R^2$.
\end{rem}
\subsection{Hypersurfaces in $\mathbb{S}^2\times\R^2$}\label{section S3R S2R2}
Let us assume that $M$ is a 3-dimensional manifold. The aim is to obtain the characterization of an immersion of $M$ in $\mathbb{S}^2\times\R^2$ in terms of usual spinor fields. Suppose that $\varphi\in\Gamma(U\Sigma)$ represents the immersion of $M$ in $\mathbb{S}^2\times\R^2,$ as in Theorem \ref{th main result SmRn} (with $m=2$, $n=2$).  Let us set $\Sigma_0=\widetilde{Q}\times_{\rho}Cl^0(5).$ If $e_0^o,e_1^o,e_2^o,e_3^o,e_4^o$ is an orthonormal basis of $\R^5,$ where $e_0^o,e_1^o$ is a basis of the second factor of $\mathbb{S}^2\times\R^2$, we set $\omega:=-e_0^o\cdot e_1^o\cdot e_2^o\cdot e_3^o,$ consider the two ideals $\mathcal{I}_1:=Cl^0(5)\cdot \frac{1}{2}\left(1-\omega\right)$ and $\mathcal{I}_2:=Cl^0(5)\cdot \frac{1}{2}\left(1+\omega\right)$ of $Cl^0(5)$ and the splitting $Cl^0(5)=\mathcal{I}_1\oplus \mathcal{I}_2.$ It induces a decomposition
\begin{equation}\label{dec1 S3R}
\varphi=\varphi_1+\varphi_2\hspace{.3cm}\in\ \Sigma_1\oplus\Sigma_2
\end{equation}
with $\Sigma_1=\widetilde{Q}\times_{\rho}\mathcal{I}_1$ and $\Sigma_2=\widetilde{Q}\times_{\rho}\mathcal{I}_2.$
As in the previous section we consider the map
$$u:\hspace{.3cm}\Sigma_2\rightarrow\Sigma_1,\hspace{.2cm}\varphi_2\mapsto u(\varphi_2)=-\nu_1\cdot\varphi_2\cdot e_0^o$$
to identify $\Sigma_2$ with $\Sigma_1,$ and for
$$\Sigma'_1:=\widetilde{Q}\times_{\rho}\ Cl(4)\cdot\frac{1}{2}\left(1-\omega\right),$$
an identification
$$\Sigma'_1\rightarrow\Sigma_1,\hspace{.3cm}\psi\mapsto\psi^*$$
such that $(X\cdot\psi)^*=X\cdot\nu_1\cdot\psi^*$ for all $X\in TM\oplus E$ and $\psi\in\Sigma_1'.$ Let us set $\psi_1,\psi_2\in\Gamma(\Sigma'_1)$ such that 
\begin{equation}\label{dec2 S3R}
\psi_1^*=\varphi_1\hspace{.5cm}\mbox{and}\hspace{.5cm}\psi_2^*=u(\varphi_2).
\end{equation}
They satisfy
\begin{equation}\label{eqn psi1 S3R}
\nabla_X\psi_1=\frac{1}{2}X_1\cdot\psi_1-\frac{1}{2}B(X)\cdot\psi_1
\end{equation}
and
\begin{equation}\label{eqn psi2 S3R}
\nabla_X\psi_2=-\frac{1}{2}X_1\cdot\psi_2-\frac{1}{2}B(X)\cdot\psi_2.
\end{equation}
We traduce in the following lemma the condition expressing that $\Phi$ commutes with the product structures: it shows that $\psi_2$ (and thus $\varphi$ and therefore the immersion) is essentially determined by $\psi_1:$
\begin{lem}
For $V_1,V_2\in\Gamma(TM\oplus E)$ such that $\Phi(V_1)=e_0^o$ and $\Phi(V_2)=e_1^o$ we have $V_1\cdot \psi_1=\psi_2$ and $V_2\cdot \psi_1=-\psi_2\cdot e_0^o\cdot e_1^o,$ and therefore $V_1\cdot V_2\cdot \psi_1=\psi_1\cdot e_0^o\cdot e_1^o.$
\end{lem}
\begin{proof}
The condition expressing that $\Phi$ commutes with the product structures reads $\Phi(\mathcal{P}_2)=\{0\}\times\R^{2}$, and $V_1,V_2$ are well-defined. The proof is then identical to the proof of Lemma \ref{lem trad str prod pres S2R} above.
\end{proof}
Equation (\ref{eqn psi1 S3R}) implies that
\begin{equation}\label{eqn psi1 S3R bis}
\nabla_X\psi_1=\frac{1}{2}X_1\cdot\psi_1-\frac{1}{2}S(X)\cdot N\cdot\psi_1
\end{equation}
where $N$ and $S$ respectively denote a unit normal vector and the corresponding shape operator of $M$ in $\mathbb{S}^2\times\R^2.$ Let us write $V_1=T_1+f_1N,$ $V_2=T_2+f_2N$ and $X_1=X-\langle X,T_1\rangle V_1-\langle X,T_2\rangle V_2.$ Under the Clifford action of the volume element $-e_0\cdot e_1\cdot e_2\cdot e_3\in Cl(TM\oplus E)$ the bundle $\Sigma'_1$ splits into $\Sigma'_1=\Sigma_1^+\oplus\Sigma_1^-$. There is a $\C$-linear isomorphism 
$$\Sigma M\simeq \Sigma_1^+,\hspace{.5cm} \Psi\mapsto \Psi^*$$ 
so that $(X\cdot\Psi)^*=X\cdot N\cdot\Psi^*$ for all $X\in TM$ and $\Psi\in\Sigma M,$ where the complex structure on $\Sigma_1^+$ is given by the right-action of $e_0^o\cdot e_1^o.$ We write $\psi_1=\psi_1^++\psi_1^-$ in $\Sigma'_1=\Sigma_1^+\oplus\Sigma_1^-$ and consider $\Psi_1,\Psi_2\in\Gamma(\Sigma M)$ such that $\Psi_1^*=\psi_1^+$ and $\Psi_2^*=N\cdot \psi_1^-.$ From (\ref{eqn psi1 S3R bis}) we have
\begin{equation}\label{M3S2R2Psi1}
\nabla_X\Psi_1=-\frac{1}{2}\left(X-\langle X,T_1\rangle (T_1-f_1)-\langle X,T_2\rangle (T_2-f_2)\right)\cdot\Psi_2-\frac{1}{2}S(X)\cdot\Psi_1
\end{equation}
and
\begin{equation}\label{M3S2R2Psi2}
\nabla_X\Psi_2=-\frac{1}{2}\left(X-\langle X,T_1\rangle (T_1+f_1)-\langle X,T_2\rangle (T_2+f_2)\right)\cdot\Psi_1+\frac{1}{2}S(X)\cdot\Psi_2,
\end{equation}
together with 
\begin{equation}\label{M3S2R2normPsi}
|\Psi_1|^2+|\Psi_2|^2=1. 
\end{equation}
Moreover,
\begin{equation}\label{M3S2R2TfPsi1}
(T_1-f_1)\cdot (T_2+f_2)\cdot \Psi_1=i\Psi_1
\end{equation}
and
\begin{equation}\label{M3S2R2TfPsi2}
(T_1+f_1)\cdot (T_2-f_2)\cdot \Psi_2=i\Psi_2
\end{equation}
Conversely, the existence of two spinor fields $\Psi_1,\Psi_2\in\Gamma(\Sigma M)$ solutions of (\ref{M3S2R2Psi1})-(\ref{M3S2R2TfPsi2}) implies the existence of an isometric immersion of $M$ into $\mathbb{S}^2\times\R^2:$ we may indeed construct $\varphi\in\Gamma(U\Sigma)$ solution of (\ref{killing eqn SmRn}) from $\Psi_1$ and $\Psi_2$, just doing step by step the converse constructions; it is such that the map $\Phi:X\mapsto\langle\langle X\cdot\varphi,\varphi\rangle\rangle$ commutes with the product structures. 
\begin{rem}
Two non-trivial spinor fields $\Psi_1,\Psi_2\in\Gamma(\Sigma M)$ solutions of (\ref{M3S2R2Psi1})-(\ref{M3S2R2Psi2}) are in fact such that $|\Psi_1|^2+|\Psi_2|^2$ is a constant, and may thus be supposed so that (\ref{M3S2R2normPsi}) holds.
\end{rem}
\section{Isometric immersions in $\mathbb{H}_1^m\times \mathbb{H}_2^n$ and $\mathbb{H}_1^m\times \mathbb{R}^n$ }\label{section HmHn HmRn}
We state here without proof the analogous results for immersions in $\mathbb{H}_1^m\times \mathbb{H}_2^n$ and $\mathbb{H}_1^m\times \mathbb{R}^n,$ where $\mathbb{H}_1^m$ and $\mathbb{H}_2^n$ are spaces of constant curvature $c_1,c_2<0.$ Here $M$ still denotes a $p$-dimensional riemannian manifold and $E\rightarrow M$ a metric bundle of rank $q$ with a connection compatible with the metric and such that $p+q=m+n.$ We suppose that a product structure $\mathcal{P}$ is given on $TM\oplus E$ as in Section \ref{section main thm SmSn}. We denote by $\R^{r,s}$ the space $\R^{r+s}$ with the metric with signature
$$-\sum_{i=1}^rdx_i^2+\sum_{j=r+1}^{r+s}dx_j^2,$$
by $Cl(r,s)$ its Clifford algebra and by $Spin(r,s)$ its spin group. For immersions in $\mathbb{H}_1^m\times\mathbb{H}_2^n$ we consider the trivial bundle $\mathcal{E}_2:=M\times\R^{2,0}\rightarrow M,$ with the natural negative metric and the trivial connection, and two orthonormal and parallel sections $\nu_1,\nu_2$ of that bundle. We also consider the representation associated to the splitting $\R^{2,m+n}=\R^{2,0}\oplus\R^{0,p}\oplus\R^{0,q}$
$$\rho:\ Spin(p)\times Spin(q)\rightarrow Spin(p)\cdot Spin(q)\ \subset Spin(2,m+n)\rightarrow Aut(Cl(2,m+n))$$
and the bundles (associated to a spin structure $\widetilde{Q}:=\widetilde{Q}_M\times_M\widetilde{Q}_E$ of $TM$ and $E$)
$$\Sigma:=\widetilde{Q}\times_{\rho}Cl(2,m+n),\hspace{1cm} U\Sigma:=\widetilde{Q}\times_{\rho}Spin(2,m+n)$$
and
$$Cl(TM\oplus E\oplus\mathcal{E}_2):=\widetilde{Q}\times_{Ad}Cl(2,m+n).$$

\begin{thm}\label{th main result HmHn}
Let $B: TM \times TM \to E$ be a symmetric tensor. The following statements are equivalent:
\begin{enumerate}
\item[(i)] There exist an isometric immersion $F:M\rightarrow \mathbb{H}_1^m\times \mathbb{H}_2^n$ and a bundle map $\Phi: TM \oplus E \to T(\mathbb{H}_1^m \times \mathbb{H}_2^n)$ above $F$ such that $\Phi(X,0)= dF(X)$ for all $X\in TM,$ which preserves the bundle metrics, maps the connection on $E$ and the tensor $B$ to the normal connection and the second fundamental form of $F$, and  is compatible with the product structures.
\\
\item[(ii)] \label{thm main HmHn (2)} There exists a section $\varphi \in \Gamma(U \Sigma )$ solution of
\begin{equation}\label{killing equation HmHn}
\nabla_X \varphi = -\frac{1}{2} (\sqrt{|c_1|}\ X_1 \cdot \nu_1  +  \sqrt{|c_2|}\ X_2 \cdot \nu_2)\cdot\varphi  -  \frac{1}{2}B(X)\cdot \varphi
\end{equation}
for all $X\in TM,$ where $X= X_1 + X_2$ is the decomposition in the product structure $\mathcal{P}$ of $TM\oplus E$, such that the map
$$Z\in TM\oplus E\ \mapsto\ \langle\langle Z \cdot \varphi , \varphi \rangle\rangle\in\R^{1,m}\times\R^{1,n}$$
commutes with the product structures $\mathcal{P}$ and $\mathcal{P}'.$
\end{enumerate}
Moreover, the bundle map $\Phi$ and the immersion $F$ are explicitly given in terms of the spinor field $\varphi$ by the formulas
$$ \Phi :TM \oplus E \to T(\mathbb{H}_1^m \times \mathbb{H}_2^n),\hspace{.3cm}Z \mapsto \langle\langle Z \cdot \varphi , \varphi \rangle\rangle$$
and
\begin{equation}\label{explicit f HmHn}
F=(\frac{1}{\sqrt{|c_1|}} \langle\langle \nu_1 \cdot \varphi , \varphi \rangle\rangle , \frac{1}{\sqrt{|c_2|}} \langle\langle\nu_2 \cdot \varphi , \varphi   \rangle\rangle )\hspace{.3cm} \in\ \mathbb{H}_1^m\times\mathbb{H}_2^n.
\end{equation}
\end{thm}
For immersions in $\mathbb{H}_1^m\times\R^n$ we consider $\mathcal{E}_1:=M\times\R^{1,0}\rightarrow M,$ a parallel section $\nu_1$ of $\mathcal{E}_1$ such that $\langle \nu_1,\nu_1\rangle=-1$ and the bundles
$$\Sigma:=\widetilde{Q}\times_{\rho}Cl(1,m+n),\hspace{1cm} U\Sigma:=\widetilde{Q}\times_{\rho}Spin(1,m+n)$$
and
$$Cl(TM\oplus E\oplus\mathcal{E}_1):=\widetilde{Q}\times_{Ad}Cl(1,m+n).$$

\begin{thm}\label{th main result HmRn}
Let $B: TM \times TM \to E$ be a symmetric tensor. The following statements are equivalent:
\begin{enumerate}
\item[(i)] There exist an isometric immersion $F:M\rightarrow \mathbb{H}_1^m\times \mathbb{R}^n$ and a bundle map $\Phi: TM \oplus E \to T\mathbb{H}_1^m \times \mathbb{R}^n$ above $F$ such that $\Phi(X,0)= dF(X)$ for all $X\in TM,$ which preserves the bundle metrics, maps the connection on $E$ and the tensor $B$ to the normal connection and the second fundamental form of $F$, and is compatible with the product structures.
\\
\item[(ii)] There exists a section $\varphi \in \Gamma(U \Sigma )$ solution of
\begin{equation}\label{th main result HmRn eqn phi}
\nabla_X \varphi = -\frac{1}{2}\sqrt{|c_1|}\ X_1 \cdot \nu_1\cdot \varphi  -  \frac{1}{2}B(X)\cdot \varphi
\end{equation}
for all $X\in TM,$ where $X= X_1 + X_2$ is the decomposition in the product structure $\mathcal{P}$ of $TM\oplus E$, such that the map
$$Z\in TM\oplus E\ \mapsto\ \langle\langle Z \cdot \varphi , \varphi \rangle\rangle\in\R^{1,m}\times\R^{n}$$
commutes with the product structures $\mathcal{P}$ and $\mathcal{P}'.$
\end{enumerate}
Moreover, the bundle map $\Phi$ and the immersion $F$ are explicitly given in terms of the spinor field $\varphi$ by the formulas
$$\Phi:\hspace{.3cm}TM \oplus E \to T\mathbb{H}^m \times \mathbb{R}^n,\hspace{.3cm}Z \mapsto \langle\langle Z \cdot \varphi , \varphi \rangle\rangle.$$
and $F=(F_1,F_2)\in\mathbb{H}_1^m\times\R^n$ with
\begin{equation}\label{explicit f HmRn}
F_1= \frac{1}{\sqrt{|c_1|}}\ \langle\langle \nu_1 \cdot \varphi ,  \varphi \rangle\rangle.
\end{equation}
\end{thm}
As in the positive curvature case, it is possible to deduce a spinorial proof of the fundamental theorem of immersions theory in $\mathbb{H}_1^m\times\mathbb{H}_2^n$ or $\mathbb{H}_1^m\times\R^n$. 

\section{CMC surfaces with $H=1/2$ in $\mathbb{H}^2\times\R$}\label{section H2R}

We consider the immersion of a surface with $H=1/2$ in $\mathbb{H}^2\times\R\subset\R^{1,2}\times\R$ represented by a spinor field $\varphi$ as in Theorem \ref{th main result HmRn} (with $m=2$ and $n=1$). Let us first introduce some notation. We denote by $N$ the unit vector normal to the surface and tangent to $\HH^2\times\R,$ it is of the form $(N',\nu)$ in $\R^{1,2}\times\R,$ and by $\nu_1$ the unit vector normal to $\HH^2\times\R$ so that the immersion reads $F=(\nu_1,h)\in\HH^2\times\R.$ The function $\nu$ is \emph{the angle function} of the immersion, and we assume that it is always positive (the surface has \emph{regular vertical projection}), and the function $h:M\rightarrow\R$ is \emph{the height function} of the immersion. We fix a conformal parameter $z=x+iy$ of the surface, in which the metric reads $\mu^2(dx^2+dy^2).$ The matrix of the shape operator in the basis $\partial_x/\mu,\partial_y/\mu$ reads
\begin{equation}\label{matrix S}
S=\left(\begin{array}{cc}1/2+\alpha&\beta\\\beta&1/2-\alpha\end{array}\right)
\end{equation}
and we set the following two important quantities
$$Q_0:=-\frac{\mu^2}{2}(\alpha-i\beta)-h_z^2\hspace{.5cm}\mbox{and}\hspace{.5cm} \tau_0:=\mu^2\nu^2.$$
Following \cite{FM1} $(Q_0,\tau_0)$ are called \emph{the Weierstrass data} of the immersion, and we will see below that they appear naturally in the equations satisfied by the spinor field representing the immersion in adapted coordinates. We will then compute the hyperbolic Gauss map in terms of these data (we will recall the definition below) and we will interpret geometrically the relation between the spinor field and the hyperbolic Gauss map. Using these observations we will show that conversely the hyperbolic Gauss map and its Weierstrass data determine a family of spinor fields (parameterized by $\C$), and thus a family of immersions, a result obtained in \cite{FM1} by other methods. We will finally use this spinorial approach to describe directly the correspondence between the theories of $H=1/2$ surfaces in $\HH^2\times\R$ and in $\R^{1,2}$.

\subsection{The Clifford algebra and the Spin group of $\R^{1,3}$}
Let us consider the algebra of complex quaternions $\HH^{\C}:=\HH\otimes\C$. An element $a$ of $\HH^{\C}$ is of the form 
$$a=a_0 1+a_1I+a_2J+a_3K,\hspace{.5cm} a_0,\ a_1,\ a_2,\ a_3\in\C,$$
its complex norm is
$$H(a,a)=a_0^2+a_1^2+a_2^2+a_3^2\ \in\C$$
and its complex conjugate is the complex quaternion
$$\widehat{a}=\overline{a_0}\ 1+\overline{a_1}\ I+\overline{a_2}\ J+\overline{a_3}\ K$$
where $\overline{a_i}$ denotes the usual complex conjugate of $a_i.$ Let us associate to
$$x=x_0e_0^o+x_1e_1^o+x_2e_2^o+x_3e_3^o\ \in\R^{1,3}$$
the complex quaternion 
$$X=ix_01+x_1I+x_2J+x_3JI\ \in\HH^{\C}$$
where $JI=-IJ=-K.$ Using the Clifford map
$$x=x_0e_0^o+x_1e_1^o+x_2e_2^o+x_3e_3^o\in\R^{1,3}\mapsto\left(\begin{array}{cc}0&X\\\widehat{X}&0\end{array}\right)\in\HH^{\C}(2)$$
we easily obtain that
$$Cl(1,3)=\left\{\left(\begin{array}{cc}a&b\\\widehat{b}&\widehat{a}\end{array}\right),\ a,b\in \HH^{\C}\right\},\hspace{.5cm}Cl^0(1,3)=\left\{\left(\begin{array}{cc}a&0\\0&\widehat{a}\end{array}\right),\ a\in \HH^{\C}\right\}$$
and
$$Spin(1,3)=\left\{\left(\begin{array}{cc}a&0\\0&\widehat{a}\end{array}\right),\ a\in \HH^{\C},\ H(a,a)=1\right\},$$
i.e. the identification
$$Spin(1,3)\simeq S^3_{\C}:=\{a\in\HH^{\C}:\ H(a,a)=1\}.$$
For the sake of simplicity, we will frequently use below the natural identifications of $Cl^0(1,3)$ and $Cl^1(1,3)$ with $\HH^{\C}$. We will moreover use the models
\begin{equation}\label{models H2 R}
\HH^2=\{ix_01+x_2J+x_3JI,\ -x_0^2+x_2^2+x_3^2=-1\},\ \R:=\{x_1 I,\ x_1\in\R\}
\end{equation}
and we will decompose the special direction $I$ of the product structure in the form
$$I=T+\nu N,$$
where $T$ is tangent and $N$ is normal to the immersion, and $\nu$ is the angle function. 

\subsection{The Killing type equation in adapted coordinates}
In a fixed spinorial frame $\widetilde{s}$ above the orthonormal frame $s=(\partial_x/\mu,\partial_y/\mu,N,\nu_1),$ the spinor field is represented by $[\varphi]=g\in S^3_{\C},$ and we consider the components
$$g_1:=\frac{1}{2}(1+iI)g,\hspace{1cm} g_2:=\frac{1}{2}(1-iI)g$$
so that $g=g_1+g_2.$ Let us note that
$$\frac{1}{2}(1+iI)\frac{1}{2}(1+iI)=\frac{1}{2}(1+iI),\hspace{.5cm} \frac{1}{2}(1-iI)\frac{1}{2}(1-iI)=\frac{1}{2}(1-iI)$$
and
$$\frac{1}{2}(1+iI)\frac{1}{2}(1-iI)=0.$$
It will be convenient to consider the following norm on $\frac{1}{2}\left(1+iI\right)\HH^{\C}$: writing an element $g_1'$ belonging to $\frac{1}{2}\left(1+iI\right)\HH^{\C}$ in the form 
$$g_1'=\frac{1}{2}\left(1+iI\right)(a+bJ)$$
for some (unique) $a,b\in\C$, we define its hermitian norm $|g_1'|^2:=|a|^2-|b|^2.$ 
\begin{prop}\label{prop eqn g1p}
The component $g_1':=\sqrt{\mu}g_1$ satisfies
\begin{equation} \label{eqn g1p}
dg_1'=(\log\sqrt{\tau_0})_zdz\ g'_1+\frac{1}{\sqrt{\tau_0}}\left(Q_0dz+\frac{\tau_0}{4}d\overline{z}\right)\ J\widehat{g_1'}I,
\end{equation}
Moreover the compatibility conditions of the equation read
\begin{equation}\label{eqn Q0 tau0}
(Q_0)_{\overline{z}}=0\hspace{.5cm}\mbox{and}\hspace{.5cm}(\log\sqrt{\tau_0})_{z\overline{z}}=-\frac{1}{\tau_0}|Q_0|^2+\frac{\tau_0}{16}
\end{equation}
and $g_1'$ is such that $|g_1'|^2=\sqrt{\tau_0}.$
\end{prop}
The equation of $g'_1$ only depends on the Weierstrass data $(Q_0,\tau_0).$ In the statement and in the rest of the section we use the following notation: $z=x+Iy,$ $dz=dx+Idy,$ $\partial_z=1/2(\partial_x-I\partial_y),$ $\overline{z}=x-Iy,$ $d\overline{z}=dx-Idy$ and $\partial_{\overline{z}}=1/2(\partial_x+I\partial_y),$ i.e. the complex parameter $z$ is with respect to the complex structure $I$.
\begin{rem} 
The Abresch-Rosenberg differential is the quadratic differential $-Q_0dz^2$. It rather appears here as a 1-form. 
\end{rem}
We will need for the computations the following form of the compatibility equations for the product structure:
\begin{lem}\label{H2R lem comp equations}
In the frame $\widetilde{s}$ the product structure $T+\nu N$ reads 
\begin{equation}\label{eqn comp prod struct}
[T]+\nu [N]=\frac{2}{\mu}h_zJ+\nu I,
\end{equation}
and $h_z,$ $\mu$ and $\nu$ satisfy the relations
\begin{equation}\label{eqn dhz}
d(h_z)=\frac{2}{\mu}\mu_zh_zdz+\frac{\sqrt{\tau_0}}{2}\left(\frac{1}{2}d\overline{z}+(\alpha-I\beta)dz\right)
\end{equation}
and
\begin{equation}\label{eqn dnuz}
\nu_z=-\frac{1}{2}h_z-(\alpha-I\beta)h_{\overline{z}}.
\end{equation}
Moreover, the two components $g_1$ and $g_2$ of the spinor field are linked by
\begin{equation}\label{g1g2}
g_2=-\frac{1}{\nu}I\widehat{g_1}I+\frac{2i}{\sqrt{\tau_0}}h_zJg_1.
\end{equation}
\end{lem}
\noindent\textit{Proof of Lemma \ref{H2R lem comp equations}:} In $\widetilde{s},$ we represent the vectors $\partial_x/\mu,$ $\partial_y/\mu$ and $N$ of the basis $s$ by, respectively,  $J,$ $JI$ and $I\in\HH^{\C}.$ Since $h$ is the second component of the immersion we have $dh(X)=\langle X,T\rangle$ and $T=1/\mu^2(\partial_xh\ \partial_x+\partial_yh\ \partial_y)$, which gives 
\begin{equation}\label{comp expr T}
[T]=\frac{1}{\mu} J(\partial_x h+I\partial_y h)=\frac{2}{\mu} h_z J
\end{equation}
and (\ref{eqn comp prod struct}). Writing that $T+\nu N$ is parallel in $\HH^2\times\R$ we obtain the two equations $\nabla T-\nu S=0$ and $\langle S,T\rangle+d\nu=0.$  The first equation yields (\ref{eqn dhz}): the first term gives $[\nabla T]=d[T]-aI[T]$ with $[T]=\frac{2}{\mu} h_z J$ and where $a$ is the Levi-Civita connection form
\begin{equation}\label{eqn a dmu muz}
a=-\frac{1}{\mu}\left(\partial_y\mu\ dx-\partial_x\mu\ dy\right)=(\frac{1}{\mu}d\mu-\frac{2}{\mu}\mu_zdz)I
\end{equation}
and the second term gives
\begin{equation}\label{comp expr S}
[S]=\mu\left(\frac{1}{2}d\overline{z}+(\alpha-I\beta)dz\right)J
\end{equation}
with $\sqrt{\tau_0}=\mu\nu.$ The second equation yields (\ref{eqn dnuz}) by a computation of $\langle S,T\rangle=-\frac{1}{2}([S] [T]+[T] [S])$ in $\HH^{\C}$ using (\ref{comp expr T}) and (\ref{comp expr S}). Finally, since the spinor field preserves the product structure we have $\langle\langle (T+\nu N)\cdot\varphi,\varphi\rangle\rangle=I,$ which reads $([T]+\nu[N])\widehat{g}=gI$ with $g=g_1+g_2,$ $[T]=2/\mu\ h_zJ$ and $[N]=I;$ the component in $1/2(1+iI)\HH^{\C}$ of that expression yields $2h_z/\mu Jg_1+\nu I g_2=\widehat{g_1}I$
and (\ref{g1g2}).
\\
\\\noindent\textit{Proof of Proposition \ref{prop eqn g1p}:}
Since $X_1=X-X_2=X-\langle X,T\rangle (T+\nu N)$ with $[X]=\mu d\overline{z}J,$ $\langle X,T\rangle=dh(X)$ and $[T]+\nu [N]$ given by (\ref{eqn comp prod struct}), we have
$$[X_1]=(\mu d\overline{z}-\frac{2}{\mu}h_zdh)J-\frac{1}{\mu}dh\sqrt{\tau_0}I$$
and since 
\begin{equation}\label{norm hz2}
2/\mu\ |h_z|^2=\mu/2\ |T|^2=\mu/2\ (1-\nu^2)=\mu/2-\tau_0/{2\mu}
\end{equation}
we obtain
\begin{equation}\label{pf eqn X1}
[X_1]=(\frac{\mu}{2}d\overline{z}-\frac{2}{\mu}h_z^2dz+\frac{\tau_0}{2\mu}d\overline{z})J-\frac{1}{\mu}dh\sqrt{\tau_0}I.
\end{equation}
Using that $[\nabla\varphi]=dg-\frac{1}{2}aIg$, (\ref{comp expr S}) and (\ref{pf eqn X1}), the Killing type equation (\ref{th main result HmRn eqn phi}) with $[\nu_1]=i1$ and $[B(X)]=[S(X)]\cdot [N]$ reads
\begin{eqnarray*}
dg\ g^{-1}-\frac{1}{2}aI&=&\frac{i}{2}[X_1]+\frac{\mu}{2}\left(Hd\overline{z}+(\alpha-I\beta)dz\right)IJ\\
&=&\frac{\mu}{2}(\frac{i}{2}+HI)d\overline{z}J-\frac{i}{2\mu}dh\sqrt{\tau_0}I+\frac{1}{\mu}\left(\frac{\mu^2}{2}(\alpha-I\beta)dz+ih_z^2dzI-i\frac{\tau_0}{4}d\overline{z}I\right)IJ.
\end{eqnarray*}
We take $H=1/2$ and we multiply both sides of the equation by $1/2(1+iI)$ on the left: since
$$\frac{1}{2}(1+iI)\ I=-i\ \frac{1}{2}(1+iI),$$
the first right-hand term vanishes and we get
$$dg_1-\frac{1}{2}aIg_1=-\frac{1}{2\mu}dh\sqrt{\tau_0}g_1-\frac{1}{\mu}\left(Q_0dz+\frac{\tau_0}{4}d\overline{z}\right)IJg_2.$$
Using (\ref{g1g2}) we obtain
$$dg_1-\frac{1}{2}aIg_1=(-\frac{1}{2\mu}h_z\sqrt{\tau_0}+\frac{2}{\sqrt{\tau_0}\mu}Q_0h_{\overline{z}})dzg_1+\frac{1}{\sqrt{\tau_0}}\left(Q_0dz+\frac{\tau_0}{4}d\overline{z}\right)J\widehat{g_1}I$$
and using finally that
$$-\frac{h_z}{2\mu}\sqrt{\tau_0}+\frac{2}{\sqrt{\tau_0}\mu}Q_0h_{\overline{z}}=\frac{1}{\nu}\nu_z$$
(this is a consequence of (\ref{eqn dnuz}) and (\ref{norm hz2})) together with (\ref{eqn a dmu muz}) we get
$$dg_1=(-\frac{1}{2\mu}d\mu+(\log\sqrt{\tau_0})_zdz)g_1+\frac{1}{\sqrt{\tau_0}}\left(Q_0dz+\frac{\tau_0}{4}d\overline{z}\right)J\widehat{g_1}I$$
and $g_1'=\sqrt{\mu}g_1$ satisfies (\ref{eqn g1p}). A computation using (\ref{eqn g1p}) twice then shows that
$$0=d(dg_1')=\left(-(\log\sqrt{\tau_0})_{z\overline{z}}-\frac{1}{\tau_0}|Q_0|^2+\frac{\tau_0}{16}\right)dz\wedge d\overline{z}\ g_1'-\frac{1}{\sqrt{\tau_0}}(Q_0)_{\overline{z}}dz\wedge d\overline{z}\ J\widehat{g_1'}I,$$
which proves (\ref{eqn Q0 tau0}). For the last claim, $g=a_01+a_1I+a_2J+a_3K$ with $a_j\in\C$ is such that $gI\widehat{\overline{g}}=[T]+\nu [N]$ belongs to $\R I\oplus\R J\oplus \R K$ (recall the last step of the proof of Lemma \ref{H2R lem comp equations}), i.e. the component of $gI\widehat{\overline{g}}$ on $i1$ is zero, which yields
\begin{equation}\label{cond im a1234}
\Im m(a_0\overline{a_1})=\Im m(a_2\overline{a_3});
\end{equation}
moreover, $\nu$ is defined as the component along the direction $I$ of the vector $\langle\langle N\cdot\varphi,\varphi\rangle\rangle=\overline{g}I\widehat{g}$ normal to the immersion, which yields by a direct computation 
$$\nu=|a_0|^2+|a_1|^2-|a_2|^2-|a_3|^2.$$
Since $g_1'=\frac{1}{2}(1+iI)(\sqrt{\mu}(a_0-ia_1)+\sqrt{\mu}(a_2-ia_3)J),$ we get using (\ref{cond im a1234})
$$|g_1'|^2=\mu (|a_0-ia_1|^2-|a_2-ia_3|^2)=\mu(|a_0|^2+|a_1|^2-|a_2|^2-|a_3|^2)=\mu\nu=\sqrt{\tau_0}.$$

\subsection{The product structure and the second component of the spinor field}
 It appears that the product structure ($\nu$ and $h_z/\mu$ in (\ref{eqn comp prod struct})) may be determined independently of $g_1'$ and $g_2'$. This relies on the following key lemma of \cite{FM1}:
\begin{lem}\label{lem h}\cite[Lemma 10]{FM1}
The function $h_z$ satisfies the system
\begin{eqnarray*}
{(h_{z})}_z&=&(\log\tau_0)_zh_z-Q_0\sqrt{\frac{\tau_0+4|h_z|^2}{\tau_0}}\\
{(h_z)}_{\overline{z}}&=&\frac{1}{4}\sqrt{\tau_0(\tau_0+4|h_z|^2)}.
\end{eqnarray*}
If we fix $z_0$ and $\theta_0\in \C,$ the system admits a unique globally defined solution $h_z$ satisfying the initial condition $h_z(z_0)=\theta_0.$
\end{lem}
The system is a consequence of the compatibility equations (\ref{eqn dhz}) and (\ref{eqn dnuz}). Since the equations only depend on the Weierstrass data $(Q_0,\tau_0),$ $h_z$ only depends on these data and on the choice of the initial condition $h_z(z_0)\in \C$. Moreover, since
$$\mu=\sqrt{\tau_0+4|h_z|^2}\hspace{.5cm}\mbox{and}\hspace{.5cm}\nu=\sqrt{\frac{\tau_0}{\tau_0+4|h_z|^2}},$$
(by (\ref{comp expr T}) and since $|T|^2=1-\nu^2$ with $\tau_0=\mu^2\nu^2$) $\mu$ and $\nu$ are also determined by $(Q_0,\tau_0)$ and $h_z(z_0).$ We finally observe that the other component $g_2':=\sqrt{\mu}g_2$ of the spinor field is given by 
\begin{equation}\label{eqn g1p g2p}
g'_2=-\frac{1}{\nu}I\widehat{g_1'}I+\frac{2i}{\mu\nu}h_zJg_1'
\end{equation}
(Equation (\ref{g1g2})) and is thus determined by $g_1'$ if $h_z$ is known: so $g_1'$ determines a family of spinor fields parametrized by $\C;$ the parameter corresponds to the choice of an initial condition for the determination of the product structure.
\subsection{The hyperbolic Gauss map}
If $N$ is normal to $M$ and tangent to $\HH^2\times\R,$ $\nu_1$ is normal to $\HH^2\times\R$ in $\R^{1,3}$ and $\nu$ is the angle function of $M$ as above, it is defined in \cite{FM1} as the map
$$G=\frac{1}{\nu}(N+\nu_1).$$
It belongs to the light-cone $\{X\in\R^{1,3}:\ |X|^2=0\}$, and since $\nu=\langle N,I\rangle$ it is of the form $G=G'+I$ where $G':\C\rightarrow\HH^2$ takes values in the model (\ref{models H2 R}) of $\HH^2$. We will frequently identify $G$ and $G'$ below. In terms of the spinor field representing the immersion, since $[N]=I,$ $[\nu_1]=i1$ and $[\varphi]=g$ in $\widetilde{s}$ it is written
\begin{equation*} 
G=\frac{1}{\nu}\overline{g}(i+I)\widehat{g}=\frac{2i}{|g_1'|^2}\overline{g_1'}\widehat{g_1'}
\end{equation*}
and we see that it only depends on the component $g_1'$ of the spinor field. A direct computation using (\ref{eqn g1p}) shows that
\begin{equation}\label{eqn dG u1 u2}
\begin{split}
dG=\sqrt{\tau_0}/4\left\{\left(\left(1+4Q_0/\tau_0\right)dz+\left(1+4\overline{Q_0}/\tau_0\right)d\overline{z}\right)u_1\right.\\\left.-i\left(\left(1-4Q_0/\tau_0\right)dz-\left(1-4\overline{Q_0}/\tau_0\right)d\overline{z}\right)u_2\right\}
\end{split}
\end{equation}
where $(u_1,u_2)$ is the positively oriented orthonormal basis of $T_G\HH^2$ 
\begin{equation}\label{u1 u2 expl}
u_1=\frac{i}{\sqrt{\tau_0}}(I\widehat{\overline{g_1'}}Jg_1'+\overline{g_1'}Jg_1'I)\hspace{.5cm}\mbox{and}\hspace{.5cm}u_2=-\frac{i}{\sqrt{\tau_0}}(\widehat{\overline{g_1'}}Jg_1'I+I\widehat{\overline{g_1'}}J\widehat{g_1'})
\end{equation}
(setting $g_1'=1/2(1+iI)(a+bJ),$ direct computations show that this is indeed a positively oriented orthonormal basis of $T_G\HH^2$). Since $H(u_1,u_1)=H(u_2,u_2)=1$ and $H(u_1,u_2)=0,$ we also readily get that
\begin{equation}\label{eqn HdG}
H(dG,dG)=Q_0dz^2+\left(\frac{\tau_0}{4}+\frac{4|Q_0|^2}{\tau_0}\right)dzd\overline{z}+\overline{Q_0}d\overline{z}^2.
\end{equation}
As observed in \cite{FM1}, since $Q_0$ is a holomorphic function, $G:M\rightarrow\HH^2$ is harmonic, and following that paper we will say that $Q_0:\C\rightarrow\C$ and $\tau_0:\C\rightarrow (0,+\infty)$ such that (\ref{eqn HdG}) and (\ref{eqn Q0 tau0}) hold form the Weierstrass data of the map $G:\C\rightarrow\HH^2.$ At a regular point of $G$ (i.e. where $Q_0\neq 0$), computations show that conversely (\ref{eqn HdG}) implies the existence of a unique positively oriented orthonormal basis $(u_1,u_2)$ of $T_G\HH^2$ such that (\ref{eqn dG u1 u2}) holds. We will assume in the rest of the paper that the set of singular points of $G$ has empty interior. 
\subsection{Interpretation in terms of a principal bundle}
If $g_1'=\frac{1}{2}\left(1+iI\right)(a+bJ)$ and $g_1''=\frac{1}{2}\left(1+iI\right)(a'+b'J)$ belong to $\frac{1}{2}\left(1+iI\right)\HH^{\C},$ we consider the hermitian product $\langle g_1',g_1''\rangle=a\overline{a'}-b\overline{b'}$ and the norm $|g_1'|^2=|a|^2-|b|^2.$ We consider the set
$$\mathcal{V}:=\{g_1'\in\frac{1}{2}\left(1+iI\right)\HH^{\C}|\ |g_1'|^2>0\}$$
and the map
$$\pi:\hspace{.5cm}\mathcal{V}\rightarrow\HH^2,\hspace{.5cm} g_1'\mapsto \frac{2i}{|g_1'|^2}\overline{g_1'}\widehat{g_1'}.$$
This is the projection of a principal bundle, with group of structure $\C^*$ acting by multiplication. This bundle is moreover equipped with a natural invariant connection form $\omega_0$ given by
$${\omega_0}_{g_1'}(v)=\frac{1}{|g_1'|^2}\ \langle v,g_1'\rangle$$
for all $g_1'\in\mathcal{V}$ and $v$ tangent to $\mathcal{V}$ at $g_1'$: the horizontal distribution at $g_1'$ is the complex line orthogonal to the line $\C.g_1'$ (the fiber of $\pi$) with respect to the hermitian product $\langle.,.\rangle$ introduced above. We consider the bundle induced from the bundle $\pi:\mathcal{V}\rightarrow\HH^2$ by the Gauss map $G:\C\rightarrow\HH^2$ 
$$G^*\mathcal{V}=\{(z,g_1')\in\C\times\mathcal{V}:\ \pi(g_1')=G(z)\}$$
and the hypersurface 
$$\mathcal{H}:=\{(z,g_1')\in G^*\mathcal{V}:\ |g_1'|^2=\sqrt{\tau_0(z)}\}$$ 
with the projection $p_1:\mathcal{H}\rightarrow\C,$ $(z,g_1')\mapsto z;$ this is a $S^1$ principal bundle. If $p_2:\mathcal{H}\rightarrow\mathcal{V},$ $(z,g_1')\mapsto g_1'$ is the other projection we consider the 1-form
\begin{equation}\label{H2R def connection omega}
\omega:=-p_1^*(\log\sqrt{\tau_0})_zdz+p_2^*\omega_0.
\end{equation}
It is a connection form on the $S^1$ principal bundle $\mathcal{H}\rightarrow\C:$
\begin{itemize}
\item it takes values in $\underline{S^1}=i\R:$ $(U,V)\in T_{(z,g_1')}\mathcal{H}$ satisfies $\Re e\langle V,g_1'\rangle=\frac{1}{2}d(\sqrt{\tau_0})(U)$
and
\begin{eqnarray*}
\omega_{(z,g_1')}(U,V)&=&-(\log\sqrt{\tau_0})_zU+\frac{1}{\sqrt{\tau_0}}\langle V,g_1'\rangle\\
&=&-\frac{1}{2}\left((\log\sqrt{\tau_0})_zU-(\log\sqrt{\tau_0})_{\overline{z}}\overline{U}\right)+\frac{i}{\sqrt{\tau_0}}\Im m\langle V,g_1'\rangle\ \in i\R;
\end{eqnarray*}
\item it is $S^1$-invariant: $S^1$ only acts on the component $g_1'$ and $\omega_0$ is $S^1$-invariant;
\item it is normalized on vertical vectors:
$$\omega\left(\frac{d}{d\theta}_{|\theta=0}(z,e^{i\theta}g_1')\right)={\omega_0}_{g_1'}(ig_1')=i.$$
\end{itemize}
\begin{prop}\label{prop g1p sec para}
The component $g_1':\C\rightarrow\mathcal{V}$ of the spinor field which represents the immersion naturally identifies to a section of $\mathcal{H}\rightarrow\C,$ which is parallel with respect to the connection $\omega.$ 
\end{prop}
\begin{proof}
$g_1'$ defines a section $\sigma:z\mapsto (z,g_1'(z))$ of $\mathcal{H}\rightarrow\C$ since $\pi(g_1')=G$ and $|g_1'|^2=\sqrt{\tau_0}.$ It satisfies
$$\sigma^*\omega=-\sigma^*p_1^*(\log\sqrt{\tau_0})_zdz+\sigma^*p_2^*\omega_0=-(\log\sqrt{\tau_0})_zdz+{g_1'}^*\omega_0=0$$
since $p_1\circ\sigma=id,$ $p_2\circ\sigma=g_1'$ and by (\ref{eqn g1p}); indeed, the right-hand term of (\ref{eqn g1p}) is horizontal: if $g_1'=\frac{1}{2}(1+iI)(a+bJ),$ then $J\widehat{g_1'}I=\frac{1}{2}(1+iI)(\overline{b}+\overline{a}J)$ is orthogonal to $g_1'$ with respect to the hermitian product $\langle.,.\rangle.$ 
\end{proof}
\begin{rem}\label{rem omega plate}
The connection $\omega$ is flat since it admits a parallel section. In fact a computation shows that $\omega$ defined by (\ref{H2R def connection omega}) is flat if and only if $(Q_0,\tau_0)$ satisfy the conditions (\ref{eqn Q0 tau0}).
\end{rem}
\begin{cor}\label{cor g1p det by G}
$g_1'$ is determined by the Gauss map $G$ and its Weierstrass data $(Q_0,\tau_0)$ up to a sign. 
\end{cor}
\begin{proof}
A parallel section of $\mathcal{H}\rightarrow\C$ is unique up to the multiplication by a complex number $e^{i\theta}\in S^1,$ and since $g_1'$ is a solution of (\ref{eqn g1p}) the section $g_1'':=e^{i\theta} g_1'$ satisfies
\begin{equation}\label{eqn g1pp}
dg_1''=(\log\sqrt{\tau_0})_zdz\ g''_1+e^{2i\theta}\frac{1}{\sqrt{\tau_0}}\left(Q_0dz+\frac{\tau_0}{4}d\overline{z}\right)\ J\widehat{g_1''}I;
\end{equation}
it is a solution of (\ref{eqn g1p}) if and only if $e^{i\theta}=\pm 1.$
\end{proof}

\subsection{Surfaces with prescribed hyperbolic Gauss map}
We assume here that $G:\C\rightarrow\HH^2$ is a given map with Weierstrass data $(Q_0,\tau_0)$ and that the set of singular points of $G$ has empty interior. The following result was obtained in \cite{FM1}.
\begin{cor}\label{cor CMC H2xR}
There exists a family of CMC surfaces with $H=1/2$ in $\HH^2\times\R$ with hyperbolic Gauss map $G$ and Weierstrass data $(Q_0,\tau_0).$ The family is parameterized by $\C\times\R.$
\end{cor}
\begin{proof}
Since the connection $\omega$ is flat (Remark \ref{rem omega plate}), the principal bundle $\mathcal{H}\rightarrow\C$ admits a globally defined parallel section that we may consider as a map $\sigma:\C\rightarrow\mathcal{V}$ such that $\pi\circ \sigma=G.$ At a regular point of $G,$ since $d\sigma-(\log\sqrt{\tau_0})_zdz\sigma$ is horizontal and projects onto $d\pi(d\sigma)=d(\pi\circ \sigma)=dG$ it satisfies
\begin{equation}\label{dsigma theta Q0 tau0}
d\sigma-(\log\sqrt{\tau_0})_zdz\sigma=e^{2i\theta}\frac{1}{\sqrt{\tau_0}}\left(Q_0dz+\frac{\tau_0}{4}d\overline{z}\right)\ J\widehat{\sigma}I
\end{equation}
for some $\theta\in\R$ (indeed, the term $1/\sqrt{\tau_0}\left(Q_0dz+\tau_0/4d\overline{z}\right)J\widehat{\sigma}I$ is horizontal at $\sigma$ and a computation shows that its projection by $d\pi$ is of the form (\ref{eqn dG u1 u2}), with an orthonormal basis $(u_1',u_2')$ of $T_G\HH^2$ which is perhaps different to the basis $(u_1,u_2)$; taking $\theta$ such that $(u_1',u_2')$ matches with $(u_1,u_2),$ the right hand term of (\ref{dsigma theta Q0 tau0}) is the horizontal lift of $dG$). The compatibility conditions of (\ref{dsigma theta Q0 tau0}) imply that $\theta$ is a constant, and in view of (\ref{eqn g1pp}) the section $g_1':=e^{-i\theta}\sigma$ is a solution of (\ref{eqn g1p}). $g_1'$ is uniquely determined up to a sign, as in Corollary \ref{cor g1p det by G}. Note that this equation extends by continuity to the singular points of $G$. We then consider a product structure $h_z$ given by Lemma \ref{lem h} (there is a family of solutions, depending on a parameter belonging to $\C$), the solution $g_2'$ given by (\ref{eqn g1p g2p}) and set $g:=1/\sqrt{\mu}(g_1'+g_2'):$ it belongs to $S^3_{\C}$ (since $H(g,g)=1$ by a computation), and we consider the spinor field $\varphi$ whose component is $g$ in $\widetilde{s},$ and the corresponding immersions into $\HH^2\times\R$; they depend on $\C\times\R$ since a last integration is required to obtain $h$ from $h_z$. 
\end{proof}
\begin{rem}
The spinorial representation formula permits to recover the explicit representation formula of the immersion in terms of all the data: calculations from the representation formula $F=(i\overline{g}\widehat{g},h)$ (formula (\ref{explicit f HmRn})) lead to the expression of the immersion in terms of $G_z,$ $Q_0,$ $\tau_0$ and $h$ given in \cite[Theorem 11]{FM1}.
\end{rem}
\subsection{Link with $H=1/2$ surfaces in $\R^{1,2}$}
We first describe with spinors the immersions of $H=1/2$ surfaces in $\R^{1,2}$ and deduce that they are entirely determined by their Gauss map and its Weierstrass data. This is a proof using spinors of a result obtained in \cite{AN} with other methods. We then obtain a simple algebraic relation between the spinor fields representing these immersions and families (parametrized by $\C$) of spinor fields representing $H=1/2$ surfaces in $\HH^2\times\R$: this gives a simple interpretation of the natural correspondence between $H=1/2$ surfaces in $\R^{1,2}$ and in $\HH^2\times\R$. Since the proofs are very similar to proofs of the preceding sections, we will omit many details. We consider the model
$$Spin(1,2)=\{\alpha_01+\alpha_1I+i\alpha_2 J+i\alpha_3 JI,\ \alpha_j\in\R,\ \alpha_0^2+\alpha_1^2-\alpha_2^2-\alpha_3^2=1\}$$
which is the subgroup of $Spin(1,3)=S^3_{\C}$ fixing $I\in\R^{1,3}$ (under the double covering $Spin(1,3)\rightarrow SO(1,3)$) and thus also leaving
$$\R^{1,2}:=\{ix_01+x_2J+x_3JI,\ x_0,x_2,x_3\in\R\}\subset\R^{1,3}$$
globally invariant. By \cite{Bay1}, if $\widetilde{Q}$ is a spin structure of $M$ and $\rho:Spin(2)\rightarrow Spin(1,2)$ a natural representation, a spacelike immersion of a surface in $\R^{1,2}$ is represented by a spinor field $\psi\in\widetilde{Q}\times_{\rho}Spin(1,2)$ solution of the Killing type equation
\begin{equation}\label{killing R12}
\nabla_X\psi=-\frac{1}{2}S(X)\cdot\nu_1\cdot\psi
\end{equation}
for all $X\in TM,$ where $\nu_1$ is the upward vector normal to $M$ in $\R^{1,2}$ so that $|\nu_1|^2=-1$ and $S=\nabla\nu_1:TM\rightarrow TM$ is the shape operator. Explicitly, the immersion is a primitive of the 1-form $\xi(X)=\langle\langle X\cdot\psi,\psi\rangle\rangle.$ Fixing a conformal parameter $z=x+iy$ of the surface in which the metric reads $\mu^2(dx^2+dy^2)$ and choosing a spinorial frame $\widetilde{s}$ above $(\partial_x/\mu,\partial_y/\mu),$ the spinor field reads as a map $v:\C\rightarrow Spin(1,2)$. Assuming that the matrix of $S$ reads as (\ref{matrix S}), we set here
$$Q_0=\frac{\mu^2}{2}(\alpha-I\beta)\hspace{.5cm}\mbox{and}\hspace{.5cm} \tau_0=\mu^2.$$
\begin{prop}\label{prop R12 eqn vp}
The function $v':=\sqrt{\mu}v$ satisfies
\begin{equation}\label{R12 eqn vp}
dv'=\left\{(\log\sqrt{\tau_0})_z\ dz+\frac{i}{\sqrt{\tau_0}}\left(Q_0dz+\frac{\tau_0}{4}d\overline{z}\right)J\right\} v'.
\end{equation}
The compatibility conditions of the equation are (\ref{eqn Q0 tau0}). Moreover $|v'|^2:=H(v',v')=\sqrt{\tau_0}.$
\end{prop}
\begin{proof}
As in the proof of Proposition \ref{prop eqn g1p}, in the spinorial frame $\widetilde{s}$ the Killing type equation (\ref{killing R12}) reads
$$dvv^{-1}-\frac{1}{2}aI=\frac{\mu}{2}(Hd\overline{z}+(\alpha-I\beta)dz)iJ,$$
which, for $H=1/2$ and the definitions of $Q_0$ and $\tau_0$, gives
\begin{equation}\label{R12 eqn v}
dvv^{-1}=\frac{1}{2}aI+\frac{1}{\sqrt{\tau_0}}\left(\frac{\tau_0}{4}d\overline{z}+Q_0dz\right)iJ.
\end{equation}
Using (\ref{eqn a dmu muz}), we obtain (\ref{R12 eqn vp}). Finally, $|v'|^2=H(v',v')=\mu H(v,v)=\mu=\sqrt{\tau_0}.$
\end{proof}
We now consider the usual Gauss map $G:\C\rightarrow\HH^2$ of the surface, still with the model (\ref{models H2 R}). Since $G=\langle\langle\nu_1.\psi,\psi\rangle\rangle$ with 
$$[\nu_1]=\left(\begin{array}{cc}0&i1\\-i1&0\end{array}\right)\hspace{.3cm}\mbox{ and }\hspace{.3cm}[\psi]=\left(\begin{array}{cc}v&0\\0&\widehat{v}\end{array}\right)$$ 
it reads
\begin{equation}\label{R12 eqn G v}
G=i\ \overline{v}\ \widehat{v}.
\end{equation}
The projection 
$$\pi':Spin(1,2)\rightarrow\HH^2,\hspace{.5cm}v\mapsto i\ \overline{v}\ \widehat{v}$$ 
is a principal bundle of group of structure $H=\{\cos\theta+\sin\theta I,\theta\in\R\}$ (acting on the left) that we equip with a natural connection: we consider the decomposition of the Lie algebra $\underline{Spin(1,2)}=\mathfrak{h}\oplus \mathfrak{m}$ with $\mathfrak{h}=\R I$ and $\mathfrak{m}=iJ(\R 1\oplus\R I),$ the projection $p_1$ onto the first factor $\mathfrak{h}$ and the connection form
$$\omega_c=p_1\circ\omega_{MC}\hspace{.3cm}\in\ \Omega^1(Spin(1,2),\mathfrak{h})$$
where $\omega_{MC}=d\sigma\ \sigma^{-1}\in\Omega^1(Spin(1,2),\underline{Spin(1,2)})$ is the Maurer-Cartan form. The bundle $\pi':Spin(1,2)\rightarrow\HH^2$ and the Gauss map $G:\C\rightarrow\HH^2$ induce a bundle
$$G^*Spin(1,2):=\{(z,v)\in \C\times Spin(1,2)|\ G(z)=\pi'(v)\}$$
with the projection $p_1:G^*Spin(1,2)\rightarrow\C,$ $(z,v)\mapsto z.$ If $p_2:G^*Spin(1,2)\rightarrow Spin(1,2),$ $(z,v)\mapsto v$ is the second projection, we consider the connection form
$$\omega:=\frac{1}{2}p_1^*aI+p_2^*\omega_c.$$
The following results are similar to results obtained in the previous section: 
\begin{prop}
The component $v:\C\rightarrow Spin(1,2)$ of the spinor field representing the immersion is naturally a section of $G^*Spin(1,2)\rightarrow\C.$ It is horizontal for the connection $\omega.$
\end{prop}
\begin{proof}
This is a traduction of (\ref{R12 eqn v}), similar to Proposition \ref{prop g1p sec para}.
\end{proof}
\begin{cor}
$v$ is determined by the Gauss map $G$ and its Weierstrass data $(Q_0,\tau_0)$ up to a sign.
\end{cor}
\begin{proof}
It is analogous to the proof of Corollary \ref{cor g1p det by G}.
\end{proof}
We now suppose that a map $G:\C\rightarrow \HH^2$ is given with Weierstrass data $(Q_0,\tau_0).$ We moreover suppose that the set of singular points of $G$ has empty interior. 
\begin{cor}\cite{AN}
There exists a $H=1/2$ surface in $\R^{1,2}$ with Gauss map $G$ and Weierstrass data $(Q_0,\tau_0).$ It is unique up to a translation in $\R^{1,2}.$
\end{cor}
\begin{proof}
As in the proof of Corollary \ref{cor CMC H2xR}, a horizontal section $v$ of $G^*Spin(1,2)\rightarrow\C$ is a map such that
\begin{equation*}
dvv^{-1}=\frac{1}{2}aI+e^{2I\theta'}\frac{1}{\sqrt{\tau_0}}\left(Q_0dz+\frac{\tau_0}{4}d\overline{z}\right)iJ
\end{equation*}
for some function $\theta':\C\rightarrow\R$ which has to be constant, and $\widetilde{v}=e^{-I\theta'}v$ is a section solution of (\ref{R12 eqn v}) such that $\pi'(\widetilde{v})=G.$ Since such a solution is unique up to a sign, the spinor field $\psi$ is determined up to a sign and the immersion is unique up to a translation (since the immersion in $\R^{1,2}$ is finally obtained by the integration of the 1-form $\xi(X)=\langle\langle X\cdot\psi,\psi\rangle\rangle$). 
\end{proof}
\begin{prop}
The correspondence
\begin{equation}\label{eqn g1p v}
g_1':=\frac{1}{2}(1+iI)v'
\end{equation}
transforms a solution $v'$ of (\ref{R12 eqn vp}) such that $\pi'(v'/|v'|)=G$ to a solution $g_1'$ of (\ref{eqn g1p}) such that $\pi(g_1')=G+I,$ and vice-versa.
\end{prop}
\begin{proof}
Assuming that (\ref{R12 eqn v}) holds we compute
$$dg_1'=\frac{1}{2}(1+iI)dv'=(\log\sqrt{\tau_0})_z\ dz\frac{1}{2}(1+iI)v'+\frac{i}{\sqrt{\tau_0}}\left(Q_0dz+\frac{\tau_0}{4}d\overline{z}\right)J \frac{1}{2}(1-iI)v'.$$
For the last term, since $Iv'=\widehat{v'}I$ ($v$ belongs to $Spin(1,2)$) we observe that
$$iJ\frac{1}{2}(1-iI)v'=J\frac{1}{2}(1-iI)Iv'=J\frac{1}{2}(1-iI)\widehat{v'}I=J\widehat{g_1'}I,$$
which shows that $g_1'$ satisfies (\ref{eqn g1p}). We finally note that $|g_1'|^2=\mu$ and therefore
$$\pi(g_1')=\frac{2i}{|g_1'|^2}\overline{g_1'}\widehat{g_1'}=2i(\overline{v}\frac{1}{2}(1-iI)\widehat{v})=i\overline{v}\widehat{v}+\overline{v}I\widehat{v}.$$
The first term $i\overline{v}\widehat{v}$ is the Gauss map $G.$ Since $I\widehat{v}=vI$ and $\overline{v}v=1,$ the second term $\overline{v}I\widehat{v}$ is the constant $I.$ So $\pi(g_1')=G+I$ and $g_1'$ lies above the same (hyperbolic) Gauss map $G.$ Conversely, if now $g_1'\in 1/2(1+iI)\HH^C$ is a solution of (\ref{eqn g1p}) we consider the unique $v'=\alpha_01+\alpha_1I+i\alpha_2J+i\alpha_3K$ such that (\ref{eqn g1p v}) holds, and similar computations show that $v'$ is a solution of (\ref{R12 eqn vp}) above $G,$ which proves the proposition. 
\end{proof}
\begin{cor}
Immersions of CMC surfaces with $H=1/2$ in $\R^{1,2}$ (up to translations) correspond to 2-parameter families of immersions of CMC surfaces with $H=1/2$ in $\HH^2\times\R$ and regular vertical projection (up to vertical translations).
\end{cor}
\begin{proof}
In the correspondence (\ref{eqn g1p v}), $g_1'$ determines a family of spinor fields parametri\-zed by $\C$, and therefore a 2-parameter family of immersions in $\HH^2\times\R$ up to a vertical translation, whereas $v'$ determines a spinor field and therefore an immersion in $\R^{1,2}$ up to a translation.
\end{proof}
Note that the correspondence preserves the (hyperbolic) Gauss maps and their Weierstrass data.
\appendix
\section{Skew-symmetric operators and Clifford algebra}\label{app clifford}
We consider $\R^N$ with its standard scalar product $\langle.,.\rangle.$ If $\eta$ and $\eta'$ belong to the Clifford algebra $Cl(N),$ we set 
$$[\eta,\eta']=\eta\cdot \eta'-\eta'\cdot \eta,$$ 
where the dot $\cdot$ is the Clifford product. We denote by $(e_1,\ldots,e_N)$ the canonical basis of $\R^N.$ The next two lemmas were proved in the Appendix A of \cite{Bay}. We include the statements here for the convenience of the readers.
\begin{lem} \label{lem1 ap1}
Let $u:\R^N\rightarrow\R^N$ be a skew-symmetric operator. Then the bivector 
\begin{equation}\label{biv rep u}
\underline{u}=\frac{1}{4}\sum_{j=1}^Ne_j\cdot u(e_j)\hspace{.3cm}\in\ \Lambda^2\R^N\subset Cl(N)
\end{equation}
represents $u$ in the sense that, for all $\xi\in\R^N,$ $[\underline{u},\xi]=u(\xi).$ We also have the formula
\begin{equation}\label{biv rep u2}
\underline{u}=\frac{1}{2}\sum_{1\leq j<k\leq N}\langle u(e_j),e_k\rangle\ e_j\cdot e_k.
\end{equation}
\end{lem}
We now assume that $\R^N=\R^p\oplus\R^q,$ $p+q=N.$
\begin{lem}\label{lem3 ap1}
Let us consider a linear map $u:\R^p\rightarrow\R^q$ and its adjoint $u^*:\R^q\rightarrow\R^p.$ Then the bivector
$$\underline{u}=\frac{1}{2}\sum_{j=1}^pe_j\cdot u(e_j)\hspace{.3cm}\in\ \Lambda^2\R^N\subset Cl(N)$$
represents 
$$\left(\begin{array}{cc}0&-u^*\\u&0\end{array}\right):\hspace{.5cm}\R^p\oplus\R^q\rightarrow\R^p\oplus\R^q$$
in the sense that, for all $\xi=\xi_p+\xi_q\in\R^N,$ $[\underline{u},\xi]=u(\xi_p)-u^*(\xi_q).$ Moreover we have
\begin{equation}\label{biv u u*}
\underline{u}=\frac{1}{4}\left(\sum_{j=1}^pe_j\cdot u(e_j)+\sum_{j=p+1}^ne_j\cdot (-u^*(e_j))\right).
\end{equation}
\end{lem}
We moreover have the following
\begin{lem}\label{lem rep u wedge v} 
Consider, for $U,V\in\R^N,$ the skew-symmetric operator
\begin{eqnarray*}
U\wedge V:\ \R^N&\rightarrow&\R^N\\
W&\mapsto& \langle U,W\rangle V-\langle V,W\rangle U.
\end{eqnarray*}
It is represented by
$$\frac{1}{4} \left( U\cdot V-V\cdot U\right)\ \in \Lambda^2\R^N\subset Cl(N),$$
i.e., for all $W\in \R^N,$
$$U\wedge V\ (W)=\left[\frac{1}{4} \left( U\cdot V-V\cdot U\right),W\right].$$
\end{lem}
\begin{proof}
Let us write
\begin{eqnarray*}
(U\wedge V)(W)&=&\langle U,W\rangle V-\langle V,W\rangle U\\
&=&\frac{1}{2}\langle U,W\rangle V+\frac{1}{2}V\langle U,W\rangle-\frac{1}{2}\langle V,W\rangle U-\frac{1}{2}U\langle V,W\rangle.
\end{eqnarray*}
Using $\langle U,W\rangle=-\frac{1}{2}\left(U\cdot W+W\cdot U\right)$ and $\langle V,W\rangle=-\frac{1}{2}\left(V\cdot W+W\cdot V\right)$ we get
\begin{eqnarray*}
(U\wedge V)(W)&=&-\frac{1}{4}\left(U\cdot W\cdot V+ W\cdot U\cdot V+ V\cdot U\cdot W+V\cdot W\cdot U\right)\\
&&+\frac{1}{4}\left(V\cdot W\cdot U+ W\cdot V\cdot U+ U\cdot V\cdot W+U\cdot W\cdot V\right)\\
&=&-\frac{1}{4}\left(W\cdot U\cdot V+ V\cdot U\cdot W-W\cdot V\cdot U-U\cdot V\cdot W\right).
\end{eqnarray*}
This last expression is
$$\left[\frac{1}{4} \left( U\cdot V-V\cdot U\right),W\right]=\frac{1}{4}\left( U\cdot V-V\cdot U\right)\cdot W-W\cdot  \frac{1}{4}\left( U\cdot V-V\cdot U\right).$$
\end{proof}

\textbf{Acknowledgments:} Pierre Bayard wishes to thank Julien Roth and the \emph{Laboratoire d'Analyse et de Math\'ematiques Appliqu\'ees} at \emph{Eiffel University} for their hospitality from august 2021 to july 2022; he also wishes to thank the program PASPA-DGAPA of the \emph{National Autonomous University of Mexico} for support. 


\begin{thebibliography}{}
\bibitem{AN} K. Akutagawa, S. Nishikawa, \textit{The Gauss map and spacelike surfaces with prescribed mean curvature in Minkowski 3-space}, T\^ohoku Math. J. \textbf{42} (1990), 67-82.
\bibitem{AMM} B. Ammann, A. Moroianu \& S. Moroianu, \textit{The Cauchy problems for Einstein metrics and parallel spinors}, Comm. Math. Phys. \textbf{320:1} (2013), 173-198.
\bibitem{Bar} C. B\"ar, \textit{Extrinsic bounds for eigenvalues of the Dirac operator}, Ann. Global Anal. Geom. \textbf{16}:6 (1998), 573-596.
\bibitem{Bay1} P. Bayard, \textit{On the spinorial representation of spacelike surfaces into 4-dimensional Minkowski space}, J. Geom. Phys. \textbf{74} (2013), 289-313.
\bibitem{Bay} P. Bayard, \textit{Spinorial representation of submanifolds in $SL_n(\C)/SU(n)$}, Adv. Appl. Clifford Alg. \textbf{29}:51 (2019).
\bibitem{BLR2} P. Bayard, M.-A. Lawn \& J. Roth, \textit{Spinorial representation of submanifolds in Riemannian space forms}, Pacific J. Math. \textbf{291}:1 (2017), 51-80.
\bibitem{FM1} I. Fern\'andez, P. Mira, \textit{Harmonic maps and constant mean curvature surfaces in $\mathbb{H}^2\times\R$}, Amer. J. Math. \textbf{129} (2007), 1145-1181. 
\bibitem{Fr} T. Friedrich, \textit{On the spinorial representation of surfaces in Euclidian 3-space}, J. Geom. Physics \textbf{28}:1-2 (1998), 143-157.
\bibitem{Ko} B.G. Konopelchenko, \textit{Weierstrass representations for surfaces in 4D spaces and their integrable deformations via DS hierarchy}, Ann. Global Anal. Geom. \textbf{18}:1 (2000), 61-74.
\bibitem{K} D. Kowalczyk, \textit{Isometric immersions into products of space forms}, Geom. Dedicata \textbf{151} (2011), 1-8.
\bibitem{KS} R. Kusner \& N. Schmitt, \textit{The spinor representation of surfaces in space}, preprint, 1996, arXiv.
\bibitem{LR} M.-A. Lawn \& J. Roth, \textit{Isometric immersions of hypersurfaces in 4-dimensional manifolds via spinors}, Diff. Geom. Appl. \textbf{28}:2 (2010), 205-219.
\bibitem{LTV}  J.H. Lira,  R. Tojeiro \& F. Vit\'orio, \textit{A Bonnet theorem for isometric immersions into products of space forms}, Arch. Math. (Basel) \textbf{95}:5 (2010), 469-479.
\bibitem{Mo} B. Morel, \textit{Surfaces in $\mathbb{S}^3$ and $\mathbb{H}^3$ via spinors}, in: Actes du S\'eminaire de Th\'eorie Spectrale \textbf{23}, Institut Fourier, Grenoble (2005), 9-22.
\bibitem{Ro1} J. Roth, \textit{Spinorial characterization of surfaces into 3-dimensional homogeneous manifolds}, J. Geom. Physics \textbf{60} (2010), 1045-1061.
\bibitem{Ro2} J. Roth, \textit{Spinors and isometric immersions of surfaces in 4-dimensional products}, Bull. Belgian Math. Soc. - Simon Stevin \textbf{21}:4 (2014), 635-652.
\bibitem{Ta} I.A. Taimanov, \textit{Surfaces of revolution in terms of solitons}, Ann. Global Anal. Geom. \textbf{15}:5 (1997), 419-435.
\bibitem{Tr} A. Trautman, \textit{Spin structures on hypersurfaces and the spectrum of the Dirac operator on spheres}, Spinors, twistors, Clifford algebras and quantum deformations (Sob\'otka Castle, 1992), Fund. Theories Phys. \textbf{52}, 25-29, Kluwer Acad. Publ., Dordrecht, 1993.
\bibitem{Va} V.V. Varlamov, \textit{Spinor representations of surfaces in 4-dimensional pseudo-Riemannian manifolds} (2000), arXiv:math/0004056.
\bibitem{Vo} L. Voss, \textit{Eigenwerte des Dirac-Operators auf Hyperfl\"achen}, Diplomarbeit, Humboldt Universit\"at zu Berlin (1999).
\end{thebibliography}
\end{document}